\documentclass{amsart}
\usepackage{amscd}
\usepackage{enumerate,amssymb,amsmath,graphics,epsfig}
\usepackage[colorlinks]{hyperref}
\usepackage{color}
\usepackage{tcolorbox}
\usepackage{mathrsfs}
\usepackage{epstopdf}
\usepackage{booktabs}
\usepackage{subfigure}
\usepackage{marginnote}
\usepackage{rotating}
\usepackage{todonotes}

\theoremstyle{plain}
\newtheorem{thm}{Theorem}
\newtheorem{proposition}[thm]{Proposition}
\newtheorem{lemma}[thm]{Lemma}

\newtheorem{problem}{Problem}
\newtheorem{ass}{Assumption}
\theoremstyle{remark}

\newtheorem{remark}{Remark}

\newcommand{\m}[1]{\mathsf{#1}}
\newcommand{\RE}{\mathbb{R}}
\newcommand\Huo{H^1_0(\Omega)}
\newcommand\Ld{L^2(\Omega)}

\newcommand\Kb{\mathbb{K}_{b,h}}
\newcommand\No{\Kb^\perp}
\newcommand\T{\mathcal{T}}

\newcommand\Pinabla{{\Pi_k^\nabla}}
\newcommand\Pio{\Pi^0_k}
\newcommand\Pinablau{{\Pi_1^\nabla}}
\newcommand\Pinablad{{\Pi_2^\nabla}}
\newcommand\Piou{\Pi^0_1}
\newcommand\Piod{\Pi^0_2}
\newcommand\Th{T_h}
\newcommand\Tht{\widetilde{T}_h}
\renewcommand\P{\mathbb{P}}
\newcommand\calM{\mathcal{M}}
\newcommand\Vh{V_h}
\newcommand\VhE{V_h(E)}
\newcommand\VhkE{V_h^k(E)}

\newcommand\Vhk{V_h^k}

\newcommand\aE{a^E}
\newcommand\bE{b^E}
\newcommand\aEh{a^E_h}
\newcommand\bEh{b^E_h}
\newcommand\SE{S^E}
\newcommand\ah{a_h}
\newcommand\bh{b_h}

\newcommand\vh{v_h}
\newcommand\uh{u_h}
\newcommand\wh{w_h}
\newcommand\fh{f_h}
\newcommand\new{\begin{color}{blue}}
\newcommand\wen{\end{color}}

\begin{document}

\title[]
{Virtual element approximation of eigenvalue problems: is the stabilization of the right hand side necessary?}
\author{Daniele Boffi}
\address{King Abdullah University of Science and Technology (KAUST), Saudi
Arabia, Dipartimento di Matematica “F. Casorati”, Università di Pavia, Italy,
IMATI-CNR ``Enrico Magenes'', Pavia, Italy}
\email{daniele.boffi@kaust.edu.sa}
\urladdr{http://cemse.kaust.edu.sa/people/person/daniele-boffi}
\author{Francesca Gardini}
\address{Dipartimento di Matematica ``F. Casorati'', Universit\`a di Pavia, Italy}
\email{francesca.gardini@unipv.it}
\urladdr{http://www-dimat.unipv.it/gardini/}
\author{Lucia Gastaldi}
\address{DICATAM, Universit\`a di Brescia, Italy, IMATI-CNR ``Enrico Magenes'', Pavia, Italy}
\email{lucia.gastaldi@unibs.it}
\urladdr{http://lucia-gastaldi.unibs.it}

\begin{abstract}
The VEM approximation of eigenvalue problems usually involves the appropriate
tuning of stabilization parameters, unless self-stabilizing or
stabilization-free VEM are used. In this paper we prove that for elliptic
self-adjoint eigenvalue problems the stabilization of the mass matrix is not
necessary when lower order standard VEM spaces are adopted. Numerical evidence
shows that also for higher order schemes the same result is true on various
mesh sequences.
\end{abstract}
\maketitle

\section{Introduction}
\label{se:intro}

This paper deals with the virtual element approximation of eigenvalue
problems associated with partial differential equations. A typical variational
formulation for an eigenvalue problem seeks eigenvalues $\lambda$
and non vanishing eigenfunctions $u$ in a suitable functional space $V$ such
that
\[
a(u,v)=\lambda b(u,v)\qquad\forall v\in V.
\]
In this exploratory work we consider an eigenvalue problem associated with a
self-adjoint elliptic operator. At the continuous level in most cases the
bilinear forms $a(\cdot,\cdot)$ and $b(\cdot,\cdot)$ are symmetric and
coercive. Standard conforming Galerkin methods lead to a matrix problem of the
form
\[
\m{A}\m{x}=\lambda\m{B}\m{x}
\]
with $\m{A}$ and $\m{B}$ symmetric and positive definite.

Several discretization schemes require suitable stabilizations of the bilinear
forms $a(\cdot,\cdot)$ and $b(\cdot,\cdot)$. The virtual element method (VEM)
is one of those and we refer to~\cite{auto-stab} for a throughout discussion
of the risks originating from the modification of the bilinear forms. When the
matrices $\m{A}$ and $\m{B}$ contain parametric terms, spurious eigensolutions
are added to the spectrum and may pollute the results if the stabilization
parameters are not chosen appropriately.
The use of standard VEM spaces seems to require such stabilization: the
bilinear forms $a(\cdot,\cdot)$ and $b(\cdot,\cdot)$ are usually evaluated
using projection operators on polynomial spaces, ensuring consistency but
leading to unstable methods where the involved matrices $\m{A}$ and $\m{B}$
could even become singular if not properly stabilized.
A look at the relevant analysis presented in~\cite{auto-stab} shows that a
naive idea that the stabilization parameters should be \emph{large enough} is
prone to subtle drawbacks, more evident for $\m{B}$ than for $\m{A}$.
Actually, increasing the stabilization parameter for the bilinear form
$a(\cdot,\cdot)$ has the effect of shifting the spurious eigenvalues to the
higher part of the spectrum, while increasing the stabilization parameter for
the bilinear form $b(\cdot,\cdot)$ may move the spurious eigenvalues towards
the lower part of the spectrum, thus polluting the typical window of interest
when eigenvalues close to the fundamental mode are sought.

The first papers dealing with VEM approximation of eigenvalue
problems~\cite{Fra-VEM,fra_nonconf,mora_steklov} show the convergence of the
numerical scheme, as the meshsize $h$ tends to zero, when the stabilizing
parameters are fixed. As explained above, it became apparent
soon~\cite{auto-stab,vem-springer} that asymptotic a priori estimates with a
prededermined choice of the parameters are not enough to guarantee that the
method is effective in practice. Already in~\cite{Fra-VEM} it can be seen that
spurious eigenvalues might appear with a wrong choice of the parameters.

In this context, several research directions are emerging, aiming at
identifying VEM schemes that are not affected by stabilization issues when
applied to eigenvalue problems. For instance, stabilization free or
self-stabilizing methods have been intensively studied during the last years
for the approximation of the source
problem~\cite{berrone2023stabilizationfree,Berronefirst,self-stabilized}. This
could be a way of avoiding the tuning of parameters, even if questions remain
open about the optimal choice of the projection degree and the increased
polynomial degree of the projections may lead to a higher assembly time and to
a worse condition number of the involved matrices~\cite{Paola}. Research in
this direction are presented, for instance,
in~\cite{Mengetal,ChenSukumar,MarconMora,Mengetal2}. In other cases, it has
been shown that schemes can work even if the right hand side matrix $\m{B}$ is
not stabilized~\cite{Linda}.

In any case, it is timely and interesting to discuss whether the stabilization
is needed when standard VEM spaces are used for the approximation of
eigenvalue problems. This is the aim of this paper and we focus on the
stabilization of the right hand side matrix $\m{B}$ that, as explained above,
is more critical than the stabilization of $\m{A}$.

In Section~\ref{se:setting} we describe the eigenvalue problem we are dealing
with, and we introduce our notation. We discuss the Descloux--Nassif--Rappaz
theory for the convergence analysis of non compact operators~\cite{DNR1}. It
is not so uncommon to use this theory when non-conforming approximations of
compact operators are present. Moreover, the standard
analysis~\cite{Boffi-acta} cannot always be used in case of VEMs due to the
difficulty of defining the discrete solution operator for all functions in
$L^2(\Omega)$.
The developed analysis holds in general and is
applied to VEM discretizations in Sections~\ref{se:vem} and~\ref{se:practice}.
Section~\ref{se:analysis} is the core of our contribution where we show that
for lower order degrees ($k=1,2$) the discrete eigensolutions converge even
when the right hand side matrix $\m{B}$ is not stabilized. Finally, a series
of numerical tests is reported in Section~\ref{se:num} for several choices of
mesh sequences. It can be appreciated that the method can converge also when
the matrix $\m{B}$ has a non trivial kernel. Moreover, it can be seen that the
convergence holds also beyond the theoretical results when higher order
schemes are used.
\section{Abstract setting}
\label{se:setting}
Let us consider two Hilbert spaces $V$ and $H$ with $V\subset H$, with dense and continuous embedding, endowed with norms $\|\cdot\|_V$ and $\|\cdot\|_H$, respectively. We introduce two symmetric and continuous bilinear forms $a:V\times V\to\RE$ and $b:H\times H\to\RE$ satisfying the following assumptions:
\begin{itemize}
\item 
there exists a positive constant $\alpha$ such that for all $v\in V$
\begin{equation}
\label{eq:coercivity}
a(v,v)\ge\alpha \|v\|_V^2;
\end{equation} 
\item for all $u\in H$ with $u\neq 0$
\begin{equation}
\label{eq:positivity}
b(u,u)>0.
\end{equation}
\end{itemize}
We consider the following eigenvalue problem. 
\begin{problem}
\label{pb:eig-pb}
Find $\lambda\in\RE$ such that there exists $u\in V$ with $u\neq 0$ satisfying
$$
a(u,v)=\lambda b(u,v)\quad\forall v\in V.
$$
\end{problem}
The associated source problem reads as follows.
\begin{problem}
\label{pb:source-pb}
Given $f\in H$ find $u\in V$ such that
$$
a(u,v)= b(f,v)\quad\forall v\in V.
$$
\end{problem}
Thanks to the coercivity and continuity assumptions, the Lax--Milgram lemma implies that Problem~\ref{pb:source-pb} is well-posed in the sense that there exists a unique solution $u\in V$ such that 
$$ 
\|u\|_V\le C\|f\|_H.
$$

We denote by $V_0$ be the space containing the solutions of Problem~\ref{pb:source-pb}
for all $f\in H$, so that we have also the regularity estimate
\begin{equation}
\label{eq:V0}
\|u\|_{V_0}\le C\|f\|_H.
\end{equation}

The analysis of variationally posed eigenvalue problems typically relies on
the definition of the solution operator $T:V\to V$ such that, for all $f\in
V$, $Tf=u$ with $u$ solution of Problem~\ref{pb:source-pb}, that is 
$$
a(Tf,v)=b(f,v)\quad\forall v\in V.
$$
We observe that $T$ is self-adjoint. We have that $\mu\in\RE$ is an eigenvalue of $T$ if there exists $u\in V$ with $u\neq 0$ such that $Tu=\mu u$. 
We assume that $T:V\to V$ is a compact operator. Hence, the spectrum of $T$ contains $0$ and 
a countable set of strictly positive eigenvalues $\{\mu_n\}_{n\in\mathbb{N}}$
(counted with their multiplicities) having only $0$ as possible accumulation
point. If the range of $T$ is not finite dimensional then the eigenvalues can
be sorted as a decreasing sequence converging to $0$ as $n$ goes to infinity.
It is well-known that these positive eigenvalues are the reciprocal of the
eigenvalues of Problem~\ref{pb:eig-pb}, that is $\mu_n=1/\lambda_n$ and 
$$
0<\lambda_1\le \lambda_2\le\dots \le\lambda_n\le\cdots, 
$$
and that the eigenspaces are the same. 
Moreover, the corresponding eigenfunctions $u_n$ are orthogonal with respect
to both the forms $a$ and $b$ (this is automatic for simple eigenvalues, while
it can be enforced in case of multiple ones), and we set that $b(u_n,u_n)=1$.

For the discretization of Problem~\ref{pb:eig-pb}, let $V_h$ be a finite dimensional subspace of $V$, and let us consider 
discrete bilinear forms $a_h:V_h\times V_h\to\RE$ and $b_h:V_h\times
V_h\to\RE$. We assume that $a_h$ and $b_h$ are continuous with respect to the
norms $\|\cdot\|_V$ and $\|\cdot\|_H$, respectively, and that there exists a positive constant $\underline{\alpha}$ independent of $h$ such that 
\begin{equation}
\label{eq:disc-coercivity}
a_h(\vh,\vh)\ge\underline{\alpha} \|\vh\|_V^2\quad\forall \vh\in V_h. 
\end{equation} 
With the above definitions, the discrete counterpart of Problem~\ref{pb:eig-pb} reads as follows. 
\begin{problem}
\label{pb:disc-eig-pb}
Find $\lambda_h\in\RE$ such that there exists $\uh\in V_h$ with $\uh\neq 0$ satisfying
$$
a_h(\uh,\vh)=\lambda_h b_h(\uh,\vh)\quad\forall \vh\in V_h.
$$
\end{problem}
Contrary to the continuous case, we do not require that $b_h$ is strictly
positive (see assumption in~\ref{eq:positivity}) but only non negative.
Therefore,
there might exist elements $\wh$ in $V_h$ such that $b_h(\wh,\vh)=0\quad\forall
\vh\in V_h$. We call the set of such $\wh$'s the kernel of $b_h$ and denote it as follows:
\begin{equation}
\label{eq:kernel-bh}
\Kb=\{\wh\in V_h: b_h(\wh,\vh)=0\quad\forall \vh\in V_h\}.
\end{equation}
The generalized algebraic eigenvalue system associated to Problem~\ref{pb:disc-eig-pb} is 
\begin{equation}
\label{eq:algebraic}
\m{A}\m{x}=\m{\lambda}\m{B}\m{x}, 
\end{equation}
with $\m{A}$, $\m{B}$ matrices of dimension $N_h:=\dim(V_h)$, and $\m{x}\in\RE^{N_h}$. The matrix $\m{B}$ may be not full rank, 
due to the fact that $b_h$ can have a non trivial kernel $\Kb$. In such case
$\m{\lambda}=\infty$ is an eigenvalue of~\eqref{eq:algebraic} 
with multiplicity equal to the dimension of $\Kb$. In our analysis, we are
going to avoid dealing with such infinite eigenvalues by restricting our
problem to a suitable subspace of $V_h$ which does not contain elements of
$\Kb$.

Similarly to the continuous case, we consider the associated discrete source problem and discrete solution operator. 
\begin{problem}
\label{pb:disc-source-pb}
Given $f_h\in V_h$, find $\uh\in V_h$ such that
$$
a_h(\uh,\vh)= b_h(f_h,\vh)\quad\forall \vh\in V_h.
$$
\end{problem}
Thanks to the coercivity assumption~\eqref{eq:disc-coercivity}, there exists a
unique solution to Problem~\ref{pb:disc-source-pb}. We observe that if
$f_h\in\Kb$, then the right hand side $b_h(f_h,\vh)$ vanishes for all $\vh\in V_h$, so
that the solution $\uh$ is zero. 

The discrete solution operator $\Tht: V_h\to V_h$, is then defined for all
$f_h\in V_h$ by $\Tht f_h=\uh$, with $\uh$ the unique solution to Problem~\ref{pb:disc-source-pb}, that is 
\begin{equation}
\label{eq:Tht}
a_h(\Tht f_h,\vh)= b_h(f_h,\vh)\quad\forall \vh\in V_h.
\end{equation}
The eigenvalues $\widetilde\mu_h\in\RE$ of $\Tht$ satisfy that there exists
$\widetilde{u}_h\in V_h$, with $\widetilde{u}_h\neq 0$, such that
$\Tht\widetilde{u}_h=\widetilde\mu_h\widetilde{u}_h$. Since $V_h$ is of
dimension $N_h$, we have exactly $N_h$ eigenvalues of 
$\Tht$. We can have that $\Kb\neq\{0\}$, so that $\Tht$ admits the null eigenvalue with multiplicity equal to $\dim(\Kb)$ and $\Kb$ is the associated eigenspace. If $\widetilde\mu_h\neq 0$, then $\widetilde{u}_h$ does not belong to $\Kb$ and 
$\left(\lambda_h,\widetilde{u}_h\right)$ with $\lambda_h=\dfrac{1}{\widetilde\mu_h}$ is an eigensolution of Problem~\ref{pb:disc-eig-pb}, 
that is 
\begin{equation}
\label{eq:tilde-eigmodes}
a_h(\widetilde{u}_h,\vh)=\lambda_h b_h(\widetilde{u}_h,\vh)\quad\forall \vh\in V_h.
\end{equation}
Taking $\vh\in\Kb$ in~\eqref{eq:tilde-eigmodes}, we obtain that
$a_h(\widetilde{u}_h,\vh)=0$, therefore $\widetilde{u}_h$ is orthogonal 
to $\Kb$ with respect to the scalar product induced by $a_h$. We denote by $\No$ such space, with the following definition
\begin{equation}
\label{eq:orth-kb}
\No=\{\wh\in V_h: a_h(\wh,\vh)=0\quad\forall \vh\in\Kb\}.
\end{equation}
We characterize the action of $\Tht$ in the following lemma. 
\begin{lemma}
\label{le:chenonhaunriferimento}
The operator  $\Tht$ maps the space $\No$ into itself. 
\end{lemma}
\begin{proof}
Let us consider $f_h\in\No$, then by definition of $\Tht$, see~\eqref{eq:Tht}, we
have that for all $\vh\in\Kb$
$$
a_h(\Tht f_h,\vh)= b_h(f_h,\vh)=0, 
$$
hence due to~\eqref{eq:orth-kb} $\Tht f_h\in\No$.
\end{proof}
The restriction of $\Tht$ to the space $\No$ is denoted by $\Th:\No\to\No$ and is defined as:
\begin{equation}
\label{eq:Th}
\forall f_h\in\No\text{ find }\Th f_h\in\Vh:\quad a_h(\Th f_h,\vh)=b_h(f_h,\vh)\quad\forall \vh\in\No.
\end{equation}
The eigenmodes of the operator $\Th$ are such that $\Th \uh=\mu_h \uh$.  Hence
$\Th$ admits exactly $N_h-\dim(\Kb)$ strictly positive eigenvalues $\mu_h$.
The pair $(\lambda_h, \uh)$ with $\lambda_h=\dfrac{1}{\mu_h}$ is an eigenmode of the following variational 
discrete eigenvalue problem.
\begin{problem}
\label{pb:perp-eig-pb}
Find $\lambda_h\in\RE$ such that there exists $\uh\in\No$ with $\uh\neq 0$ satisfying
$$
a_h(\uh,\vh)=\lambda_h b_h(\uh,\vh)\quad\forall \vh\in\No.
$$
\end{problem}
In order to discuss the convergence of the discrete eigenmodes of
Problem~\ref{pb:perp-eig-pb} to those of Problem~\ref{pb:eig-pb}, 
we use the abstract theory of~\cite{DNR1}. In particular, we consider
the following two properties: 
\begin{description}
\item[P1] $\displaystyle\sup_{\fh\in\No}\dfrac{\|(T-\Th)\fh\|_{V}}{\|\fh\|_V}$
tends to $0$ as $h$ goes to $0$;
\item[P2] $\displaystyle\inf_{\vh\in\No} \| u-\vh\|_V$ tends to $0$ as $h$ goes
to $0$ for all $u\in V_0$, where $V_0$ is the space containing the solutions of
Problem~\ref{pb:source-pb}.
\end{description}

The following theorem was proved in~\cite{DNR1}:
\begin{thm}
\label{th:DNRI}
Let us assume that properties\/ $\mathbf{P1}$ and\/ $\mathbf{P2}$ are satisfied, then the
eigenvalues of Problem~\ref{pb:perp-eig-pb} converge to those of
Problem~\ref{pb:eig-pb} in the sense that for any $\lambda$ of multiplicity
$m$ solution of Problem~\ref{pb:eig-pb} there exist exactly $m$ discrete
eigenvalues solution of Problem~\ref{pb:perp-eig-pb} converging to it as $h\to 0$.
\end{thm}

\begin{remark}
An analogous theorem holds also for the convergence of the eigenspaces. As
usual, in case of a multiple eigenvalue $\lambda$, the discrete eigenspaces
associated with all the eigenvalues converging to $\lambda$ should be
considered.
\end{remark}

\begin{remark}
In this paper we address the issue of the convergence (and absence of spurious
modes) without investigating the rate of convergence. Actually, once the
convergence is assured, estimating the rate can be done, for instance, using
the tools of~\cite{DNR2} and is usually an easier task.
\end{remark}

\section{Property P1 for general VEM approximation of elliptic eigenvalue
problem}
\label{se:vem}
As an example for the situation illustrated in the above section, we consider
the Virtual Element Method (VEM) for approximating an elliptic
partial differential equation. In
this case the bilinear forms $a_h$ and $b_h$ have to be chosen carefully and
stabilization techniques are usually introduced in order that the
matrices in~\eqref{eq:algebraic} are positive definite.
However, the stabilization procedure introduces parameters in the discrete
formulation of the problem and optimal choices of them could be a difficult
task, as it has been pointed out in~\cite{auto-stab}. Actually the main
troubles arise in connection with the choice of the stabilization parameter of
$b_h$, as it can be seen in~\cite[Fig.~7]{Fra-VEM}. Therefore here we
discuss the VEM discretization of the Poisson equation avoiding the
stabilization of $b_h$.

A typical virtual element method relies on a polygonal decomposition of the
domain. 
Let $\Omega\subset\RE^2$ be an open connected polygonal domain and let us
introduce a family of decomposition $\T_h$ of $\Omega$ into non overlapping
polygons $E$ of an arbitrary number of edges satisfying the following standard
assumptions. We denote by $h_e$ the length of the edge $e$ for
$e\subset\partial E$, by $h_E$ the diameter of $E$ and by $h=\max\{h_E,\
E\in\T_h\}$ the mesh size.
We assume that there exists a constant $\delta_1>0$ independent of $h$ such that
$h\le\delta_1 h_e$ for all $e\subset\partial E$ with $E\in\T_h$.
Each $E$ is star shaped with respect to a ball with radius $\rho_E$, and
there exists $\delta_2>0$ such that $\rho_E\ge\delta_2 h_E$. The notation
$\P_k(E)$ stands for the space of polynomials of degree at most $k$ on the
element $E$, while $\P_k(\T_h)$ stands for the space of piecewise polynomials
of degree at most $k$ on the mesh $\T_h$.

In order to deal with the virtual element discretization we need that the
Hilbert spaces $H$ and $V$ contain functions defined on $\Omega$ and enjoy
some local properties as follows. For
any subset $D$ of $\Omega$, both $H$ and $V$ can be restricted to $D$.
We denote such restriction $H(D)$ and $V(D)$, and the corresponding 
norms $\|\cdot\|_{H(D)}$ and $\|\cdot\|_{V(D)}$, respectively. Similarly, the
bilinear forms $a$ and $b$ can be restricted to any subdomain $D\subset\Omega$.

Following~\cite{basic}, we introduce the basic properties of the virtual
element discretization for a general elliptic problem (see
Problem~\ref{pb:source-pb}).

We consider a finite dimensional subspace $V_h\subset V$ constructed in such a
way that its restriction $\VhE$ to each element $E\in\T_h$ contains $\P_k(E)$
with $k\ge1$.
Then, we introduce the continuous bilinear forms $a_h$ and $b_h$ from
$\Vh\times \Vh$ to $\RE$ which are continuous in $V$ and in $H$, respectively.
Moreover they satisfy the following local properties
\begin{equation}
\label{eq:VEM_global}
\aligned
&\ah(\vh,\wh)=\sum_{E\in\T_h}\aEh(\vh,\wh)\quad\forall \vh,\wh\in\Vh\\
&\bh(\vh,\wh)=\sum_{E\in\T_h}\bEh(\vh,\wh)\quad\forall \vh,\wh\in\Vh.
\endaligned
\end{equation}
The local bilinear form $\aEh(\cdot,\cdot)$ is assumed to be consistent and
stable, that is:
\begin{equation}
\label{eq:consistency}
\aEh(\uh,\vh)=\aE(\uh,\vh)
\end{equation}
whenever $\uh$ or $\vh$ belongs to $\P_k(E)$, and
\begin{equation}
\label{eq:stabilization}
\underline C\aE(\vh,\vh)\le \aEh(\vh,\vh)\le\overline C\aE(\vh,\vh)
\end{equation}
for all $\vh\in\VhE$.

From the stability property we have that $a_h$ is coercive on $V$ since it
inherits this property from the continuous bilinear form $a$. This together
with the continuity of $b_h$ implies that there exists a unique solution of the
discrete source Problem~\ref{pb:disc-source-pb} with
\[
\|\uh\|_V\le C\|f\|_H.
\]

The following approximation properties of the VEM space are useful for showing
the convergence of the discrete solution to the continuous one. We recall that
$V_0$ is the subspace of $V$ containing the solutions of
Problem~\ref{pb:source-pb} for $f\in H$. We denote by $V_0(D)$ the
subspace of the restrictions of the elements in $V_0$ to any $D\subset\Omega$
endowed with norm $\|\cdot\|_{V_0(D)}$. 
\begin{ass}
\label{as:pi-proj}
There exists $\omega_1(h)$ tending to $0$ as $h\to0$, such that for $w\in V_0$
there exists a piecewise polynomial element $w_\pi$ with $w_\pi|_E\in\P_k(E)$ satisfying
\[
\|w-w_\pi\|_{V(E)}\le\omega_1(h)\|w\|_{V_0(E)}
\]
for all $E\in\T_h$.
\end{ass}
\begin{ass}
\label{as:interp}
There exists $\omega_2(h)$ tending to $0$ as $h\to0$, such that for $w\in V_0$
one can find an interpolant $w^I\in\Vh\subset V$ satisfying
\[
\|w-w^I\|_{V(E)}\le\omega_2(h)\|w\|_{V_0(E)}
\]
for all $E\in\T_h$.
\end{ass}

The next assumption takes care of the consistency error that we might introduce
using the bilinear form $b_h$ instead of $b$.

\begin{ass}
\label{as:cons-b}
There exists $\omega_3(h)$ tending to $0$ as $h\to0$, such that for any
$f\in V$
\[
\sup_{\vh\in\Vh}\frac{b(f,\vh)-\bh(f,\vh)}{\|\vh\|_{V}}
\le\omega_3(h)\|f\|_V.
\]
\end{ass}
With the above abstract setting we are now in the
position of proving property $\mathbf{P1}$ for $T$ and $T_h$ defined in
Section~\ref{se:setting}.
\begin{proposition}
\label{Pr:P1}
If Assumptions~\ref{as:pi-proj}--\ref{as:cons-b} hold true, then
there exists $\omega(h)$ tending to $0$ as $h\to0$, such that for all
$\fh\in\No$ it holds 
\[
\|(T-\Th)\fh\|_{V}\le\omega(h)\|\fh\|_{V}.
\]
\end{proposition}
\begin{proof}
Given $\fh\in\No$, let us consider $u=T\fh$ and $\uh=\Th\fh$ which are the solution of the
following equations
\begin{equation}
\label{eq:source-fh}
\aligned
&a(u,v)=b(\fh,v) \qquad \forall v\in V\\
&\ah(\uh,\vh)=\bh(\fh,\vh) \qquad \forall\vh\in\Vh
\endaligned
\end{equation}
Let $u^I\in\Vh$ be the interpolant of $u$ defined in Assumption~\ref{as:interp},
then we can use the ellipticity of $\ah$ and write 
%
\[
\aligned
\alpha_*\|\uh-u^I\|_V^2&\le\ah(\uh-u^I,\uh-u^I)\\
&=\ah(\uh,\uh-u^I)-\ah(u^I,\uh-u^I)\\
&=\bh(\fh,\uh-u^I)-\sum_{E\in\T_h}\aEh(u^I,\uh-u^I).
\endaligned
\]
We evaluate the last term introducing the piecewise $\P_k$ approximation
$u_\pi$ of $u$ (see Assumption~\ref{as:pi-proj}) and taking into account the
consistency of $\ah$, see~\eqref{eq:consistency}
\[
\aligned
\sum_{E\in\T_h}\aEh(u^I,\uh-u^I)
&=\sum_{E\in\T_h}\left(\aEh(u^I-u_\pi,\uh-u^I)+\aE(u_\pi,\uh-u^I)\right)\\
&\le C(\|u^I-u\|_V+\|u-u_\pi\|_h)\|\uh-u^I\|_V\\
&\quad+\sum_{E\in\T_h}\left(\aE(u_\pi-u,\uh-u^I)+\aE(u,\uh-u^I)\right)\\
&\le C(\|u^I-u\|_V+\|u-u_\pi\|_h)\|\uh-u^I\|_V+a(u,\uh-u^I),
\endaligned
\]
where 
$\|\cdot\|_h=\left(\sum_{E\in\T_h}\|\cdot\|_{V(E)}^2\right)^{1/2}$.
We put together the last two relations, using~\eqref{eq:source-fh} and we 
get
\[
\aligned
\alpha_*\|\uh-u^I\|_V^2&\le C(\|u^I-u\|_V+\|u-u_\pi\|_h)\|\uh-u^I\|_V\\
&\qquad  +\bh(\fh,\uh-u^I)-b(\fh,\uh-u^I).
\endaligned
\]
Due to Assumptions~\ref{as:pi-proj} and~\ref{as:interp}, the first term on the right
hand side is bounded by $(\omega_1(h)+\omega_2(h))\|u\|_{V_0}$. To
evaluate the difference between $\bh$ and $b$ we use Assumption~\ref{as:cons-b}
as follows
\[
\bh(\fh,\uh-u^I)-b(\fh,\uh-u^I)\le\omega_3(h)\|\fh\|_V\|\uh-u^I\|_V
\]
where we used that $\fh\in\Vh\subset V$. 

Hence we have obtained the desired estimate with
\[
\omega(h)=\omega_1(h)+\omega_2(h)+\omega_3(h).
\]
\end{proof}

\section{Virtual elements in practice}
\label{se:practice}

We consider an eigenvalue problem associated to a second order elliptic
diffusion-reaction operator. Without losing generality we restrict ourselves to  
the eigenvalue problem associated with the Laplace operator: find
$\lambda\in\RE$ such that there exists $u\in\Huo$ with $u\ne0$ satisfying
\begin{equation}
\label{eq:poisson}
(\nabla u,\nabla v)=\lambda(u,v)\quad \forall v\in\Huo.
\end{equation} 
Problem~\eqref{eq:poisson} fits into the abstract framework presented in
Section~\ref{se:setting} with $V=\Huo$, $H=\Ld$, $a(u,v)=(\nabla u,\nabla v)$ and
$b(u,v)=(u,v)$. The solution operator $T:\Huo\to\Huo$ is defined as follows:
given $f\in\Huo$, $Tf=u$ where $u$ is the solution of the equation:
\begin{equation}
\label{eq:source}
a(u,v)=(f,v)\quad\forall v\in\Huo.
\end{equation}
Since, for polygonal domains, the solution of~\eqref{eq:source}
belongs to $H^{1+s}(\Omega)$ for some $s>1/2$ we have that
$V_0\subset H^{1+s}(\Omega)$. Moreover, the following bound holds
\begin{equation}
\label{eq:regu}
\|u\|_{1+s}\le C\|f\|_0.
\end{equation}
Since $H^{1+s}(\Omega)$ is compactly
embedded into $\Huo$ the solution operator $T$ results to be compact.

\subsection{Virtual element method}
\label{se:VEM}

We consider a decomposition of $\Omega$ into polygons as described at the
beginning of the previous section. 
In each polygon $E$, we define a local virtual space for $k\ge1$ as follows,
\begin{equation}
\label{eq:localVEM}
\aligned
\VhkE=&\{\vh\in H^1(E): \vh\in C^0(\partial E),
\ \vh|_e\in\P^k(e)\ \forall e\in\partial E,\\
&\quad \Delta \vh\in\P^k(E),\ (\Pinabla \vh-\vh,q)_E=0 \ 
\forall q\in\calM^{k,k-1}(E)\}.
\endaligned
\end{equation}
Here $\Pinabla$ is the projection operator from $H^1(E)$ to $\P^k(E)\subset\VhkE$ such that
\begin{equation}
\label{eq:pinabla}
\aligned
&\aE(\Pinabla v, q)=\aE(v,q)\quad\forall q\in\P^{k}(E),\\
&(\Pinabla v, 1)_E=(v,1)_E, \text{ for } k\ge2\\
&(\Pinabla v, 1)_{\partial E}=(v,1)_{\partial E}, \text{ for } k=1,
\endaligned
\end{equation}
with $\aE(v,w)=(\nabla v,\nabla w)_E$. Moreover, $\calM^{k}(E)$ is
the subspace of $\P^k(E)$ containing monomials of degree $k$.
We refer to~\cite{enhanced} for the detailed construction of the local virtual
element space and the definition of the degrees of freedom.

\begin{remark}
By definition, the local virtual element space contains not only polynomials of
degree $k$, but also  \emph{virtual} unknown functions. 
Using the degrees of freedom, we have that $\Pinabla\vh$ is computable 
for all $\vh\in\VhkE$, so that we can evaluate $\aE(\Pinabla\uh,\Pinabla\vh)$
for all $\uh,\vh\in\VhkE$, but the associated algebraic matrix could be
singular since we can have $\Pinabla\vh=0$ for some $\vh\ne0$.

The usual stabilization of $\aE$ involves a
suitable bilinear form $\SE$ with the following property: there exist positive
constants $\underline C$ and $\overline C$ satisfying 
\begin{equation}
\label{eq:stabilization-SE}
\underline C\aE(\vh,\vh)\le \SE(\vh,\vh)\le\overline C\aE(\vh,\vh)
\end{equation}
for all $\vh$ belonging to the kernel of $\Pinabla$.

\end{remark}
Hence, we use the following discrete version of the bilinear form
$\aE$
\begin{equation}
\label{eq:aEh}
\aEh(\uh,\vh)=\aE(\Pinabla \uh,\Pinabla \vh)
              +\SE((\uh-\Pinabla\uh,\vh-\Pinabla\vh).
\end{equation}
We observe that if either $\uh$ or $\vh$ in the above equations belong to
$\P_k(E)$ then the local discrete bilinear form $\aEh$ satifies both the
consistency and stability assumptions, see~\eqref{eq:consistency}
and~\eqref{eq:stabilization}, respectively.

Analogous considerations hold for the bilinear form $\bE$. 
Setting $\bE(v,w)=(v,w)_E$, the projection operator
$\Pio:L^2(E)\to\P^k(E)\subset\VhkE$ is used as follows
\begin{equation}
\label{eq:Pio}
\bE(\Pio v,q)=\bE(v,q)\quad \forall q\in\P^k(E).
\end{equation}
We have that $\Pio\vh$ is computable for all $\vh\in\VhkE$ using the degrees
of freedom.
Then, the discrete version of $\bE$ is
\begin{equation}
\label{eq:bEh}
\bEh(\uh,\vh)=\bE(\Pio \uh,\Pio \vh) \quad \forall \uh,\,\vh\in\VhkE.
\end{equation}
Notice that, differently than the bilinear form $\aEh$, we do not add a
stabilization term to $\bEh$. 

The global VEM space is then given by:
\[
\Vhk=\{v\in\Huo: v|_E\in\VhkE\ \forall E\in\T_h\}\\
\]
In the following, with abuse of
notation, we denote by $\Pinabla$ also the global projector operator
so that, for all $\vh\in\Vhk$, $\Pinabla \vh\in\P_k(\T_h)$ coincides
with the local projector when restricted to an element. Similarly, for
$f\in\Ld$, we denote by $\Pio f$ the piecewise $L^2$-projector onto
$\P_k(\T_h)$. Therefore, the bilinear form $\bh$ is well defined also in
$\Ld\times\Ld$.

Given $\vh\in\Vhk$, its projection $\Pio\vh$ could vanish on each element $E\in\T_h$, 
and we define the kernel of $\bh$ as follows
\begin{equation}
\label{eq:kbh-VEM}
\Kb=\{\wh\in\Vhk: b_h(\wh,\vh)=0\quad\forall \vh\in\Vhk\}=\{\wh\in\Vhk: \Pio\wh=0\}.
\end{equation}
We recall basic results for the VEM approximation which show that
Assumptions~\ref{as:pi-proj} and~\ref{as:interp} are satisfied,
see ~\cite{basic,Acta-VEM}.

\begin{lemma}
\label{le:pi-proj}
Under the assumptions on the mesh $\T_h$, for $w\in H^{1+s}(E)$ with $0<s\le k$
there exists an element $w_\pi\in\P_k(E)$ satisfying
\[
\|w-w_\pi\|_{0,E}+h_E|w-w_\pi|_{1,E}\le C h_E^{s+1}|w|_{s+1,E}
\]
for a suitable constant $C>0$ independent of $h_E$.
\end{lemma}
Using the degrees of freedom defined above, one can define an interpolant
$w^I\in\Vhk$ for all $w\in H^{1+s}(\Omega)$ with $0<s\le k$ such that the
following approximation property holds true.
\begin{lemma}
\label{le:interp}
Under the assumptions on the mesh $\T_h$, there exists a constant $C>0$
independent of $h$ such that for $w\in H^{1+s}(\Omega)$ with $0<s\le k$,
there exists an interpolant $w^I\in\Vhk$ satisfying
\[
\|w-w^I\|_{0,E}+h_E|w-w^I|_{1,E}\le C h_E^{s+1}|w|_{s+1,E}\quad
\text{ for all }E\in\T_h.
\]
\end{lemma}

\subsection{Discrete eigenvalue problem and solution operator}
\label{se:VEM-Th}
With the definitions given in the previous section, the discrete counterpart of
the eigenvalue problem~\eqref{eq:poisson} reads: find $\lambda_h\in\RE$ such
that there exists $\uh\in\Vhk$ with $\uh\ne0$ satisfying
\begin{equation}
\label{eq:Poissonh}
\ah(\uh,\vh)=\lambda_h\bh(\uh,\vh)\quad\forall \vh\in\Vhk.
\end{equation}
We associate to the eigenvalue problem~\eqref{eq:Poissonh}, the discrete source
problem as follows: given $f\in\Ld$ find $\uh\in\Vhk$ such that 
\begin{equation}
\label{eq:sourceh}
\ah(\uh,\vh)=\bh(f,\vh)\quad\forall \vh\in\Vhk.
\end{equation}
It is well-known that problem~\eqref{eq:sourceh} admits a unique solution,
see~\cite{basic,Acta-VEM}. 
\begin{lemma}
\label{le:exist_discr}
For all $f\in\Ld$ there exists a unique solution $\uh\in\Vhk$ of
problem~\eqref{eq:sourceh} with the following a priori estimate
\[
\|\uh\|_{1,\Omega}\le C\|f\|_{0,\Omega}.
\]
\end{lemma}
In the following lemma we report the error estimate for the VEM approximation
of the source problem, provided by~\cite[Theorem~3.1]{basic}.
\begin{lemma}
\label{le:err-estimate}
Let $u\in\Huo$ and $\uh\in\Vhk$ be the solutions to problems~\eqref{eq:source}
and~\eqref{eq:sourceh}, respectively.
There exists a constant $C>0$ independent of $h$ such that
\[
|u-\uh|_{1,\Omega}\le C\left(|u-u^I|_{1,\Omega}+|u-u_\pi|_{h,\Omega}
+\sup_{\vh\in\Vhk}\frac{b(f,\vh)-\bh(f,\vh)}{\|\vh\|_{1,\Omega}}\right),
\]
where $|\cdot|_{h,\Omega}$ stands for the broken $H^1$-seminorm.
\end{lemma}
The discrete problem~\eqref{eq:sourceh} fits into the framework of
Section~\ref{se:setting}. Indeed, $\ah$ is coercive and continuous and $\bh$ is
continuous. In addition, the bilinear form $\bh$ might not be positive due to
its construction which makes use of the projector $\Pio$. Since we
are not adding any stabilization term, the kernel $\Kb$
(see~\eqref{eq:kernel-bh}) might not be
reduced to the $0$ element, as we shall show in the numerical results section.

Let us recall the solution operator associated to~\eqref{eq:sourceh}. In view
of the solution of the eigenvalue problem, following the discussion of
Section~\ref{se:setting}, we consider $\fh$ in 
\[
\No=\{\wh\in\Vhk: \ah(\wh,\vh)=0\ \forall \vh\in\Kb\},
\]
the orthogonal complement of $\Kb$ with respect to $\ah$, see~\eqref{eq:orth-kb}.  
Then $\Th:\No\to\Vhk$ maps elements of $\No$ to the solution of the following
equation:
\begin{equation}
\label{eq:ThVEM}
\ah(\Th \fh,\vh)=\bh(\fh,\vh)\quad\forall\vh\in\Vhk.
\end{equation}
As observed in Lemma~\ref{le:chenonhaunriferimento}, $\Th:\No\to\No$.

It remains to show that properties $\mathbf{P1}$ and $\mathbf{P2}$ hold true
for the VEM approximation of the eigensolutions of problem~\eqref{eq:poisson}.
In the following section, we shall investigate the validity of these two
properties depending on the degree of the virtual element spaces and on the
type of meshes.

\section{Convergence analysis}
\label{se:analysis}
In this section, we prove properties $\mathbf{P1}$ and $\mathbf{P2}$ for $T$
and $\Th$ defined, respectively, in Subsection~\ref{se:practice} and
in~\eqref{eq:ThVEM}. In order to show that $\mathbf{P1}$ holds true,
we apply Proposition~\ref{Pr:P1}. Since Assumptions~\ref{as:pi-proj}
and~\ref{as:interp} have been already checked, it remains to show that
Assumption~\ref{as:cons-b} is verified.
\begin{lemma}
\label{le:cons-b}
For any $f\in H^1(\Omega)$  we have
\[ 
\sup_{\vh\in\Vh}\frac{b(f,\vh)-\bh(f,\vh)}{\|\vh\|_1}
\le C h\|f\|_1.
\]
\end{lemma}
\begin{proof}
By definition of the bilinear forms $\bh$ and $b$, we have
\[
\aligned
\bh(f,\vh)-b(f,\vh)&
=\sum_{E\in\T_h}\bE(\Pio f-f,\vh)\\
&\le C \sum_{E\in\T_h}\|\Pio f-f\|_{0,E}\|\vh\|_{0,E}\\
&\le C h\|f\|_1\|\vh\|_1,
\endaligned
\]
where we used the continuity of $\bE$ and that $f\in H^1(\Omega)$. 
\end{proof}

We collect the previous results in the following statement.

\begin{thm}
For any mesh satisfying the assumptions described at the beginning of
Section~\ref{se:vem}, property $\mathbf{P1}$ holds true.
\end{thm}
The proof of property $\mathbf{P2}$ depends on the number of the edges of the
polygons of the mesh and on the degree of polynomials in the VEM space $\Vhk$.
In the following proposition, we show that property $\mathbf{P2}$ holds true for
polygons of any number of edges and for $k=1$ and $k=2$.

\begin{proposition}
\label{pr:P2-k1}
Let us consider $k=1$ and $k=2$, then for any mesh satisfying the
assumptions described at the beginning of Section~\ref{se:vem}, property
$\mathbf{P2}$ holds true.
\end{proposition}
\begin{proof}
Our proof is based on the property that  $\Pinabla\vh=\Pio\vh$ for $k=1,2$, as
it is observed in~\cite{enhanced}. In order to show this fact,
we recall the definition of the local virtual space $\VhkE$ as
in~\eqref{eq:localVEM}.
\[
\aligned
\VhkE=&\{\vh\in H^1(E): \vh\in C^0(\partial E),\ \vh|_e\in\P^k(e)\ \forall
e\in\partial E,\\
&\quad\Delta\vh\in\P^k(E),\ (\Pinabla\vh-\vh,q)_E=0\ \forall q\in\calM^{k,k-1}(E)\}.
\endaligned
\]
Therefore, for $k=1$, the definition of the local space implies that  
\[
(\Pinablau\vh-\vh,q)_E=0\quad\forall q\in\calM^{1,0}(E),
\]
which corresponds to the definition of $\Piou$ as the $L^2$-projection from
$L^2(E)$ to $\P^1(E)$. On the other hand, for $k=2$ the enhanced condition
reads
\[
(\Pinablad\vh-\vh,q)_E=0\quad\forall q\in\calM^{2,1}(E),
\]
moreover from~\eqref{eq:pinabla} we also have $(\Pinablad\vh-\vh,1)_E=0$. These
conditions ensure that $\Pinablad$ satisfies the same moment constraints as $\Piod$. 

To prove property $\mathbf{P2}$, for any $v\in V_0$ we will find an element
$\bar{v}\in\No$ such that $\|v-\bar{v}\|_1$ tends to zero as $h\to0$. We
remind that~\eqref{eq:regu} ensures that $V_0$, 
the space of the solution of the source equation~\eqref{eq:source}, is contained in $H^{1+s}(\Omega)$. 

Let $v\in V_0\subset H^{1+s}(\Omega)$, we can take $\bar{v}$ as the
interpolant $v^I\in\Vhk$ as defined in Lemma~\ref{le:interp}. 
We decompose $v^I$ into the sum $v^I=v^0+v^\perp$ with $v^0\in\Kb$ and $\ v^\perp\in\No$ 
defined as the following projections:
\begin{equation}
\label{eq:decomposition}
\aligned
& \ah(v^0,w) = \ah(v^I,w)\quad && \forall w\in\Kb\\
& \ah(v^\perp,w) = \ah(v^I,w)\quad && \forall w\in\No.
\endaligned
\end{equation}
Since $v^0\in\Kb$, then $\Pio v^0=0$ (see~\eqref{eq:kbh-VEM}) and also
$\Pinabla v^0=0$ for $k=1,2$. Hence we have 
\begin{equation}
\label{eq:ah-v0}
\aligned
\ah(v^0,\vh) & =\sum_{E\in\T_h} \left ( \aE(\Pinabla v^0,\Pinabla \vh) +
\SE((I-\Pinabla)v^0, (I-\Pinabla)\vh)\right ) \\
& = \sum_{E\in\T_h} \SE(v^0, (I-\Pinabla)\vh). 
\endaligned
\end{equation}

We now show that $\|v-v^\perp\|_1$ tends to zero as $h\to 0$. By triangular
inequality we have that 
\begin{equation}
\label{eq:triangular}
\|v-v^\perp\|_1 \le \| v-v^I\|_1 + \| v^0\|_1,
\end{equation}
hence, thanks to Lemma~\ref{le:interp}, it remains to show that $\| v^0\|_1$ tends to zero. 
Using the ellipticity of $\ah$~\eqref{eq:stabilization}, the decomposition of
$v^I$~\eqref{eq:decomposition}, 
and~\eqref{eq:ah-v0}, and the properties of the stabilization bilinear
form~\eqref{eq:stabilization-SE}, we obtain 
\begin{equation*}
\aligned
\alpha_* \|v^0\|_1^2 & \le \ah(v^0,v^0)  = \ah(v^I,v^0) =  \sum_{E\in\T_h} \SE((I-\Pinabla)v^I,v^0)\\
& \le  \sum_{E\in\T_h} \SE((I-\Pinabla)v^I,(I-\Pinabla)v^I)^{1/2} \SE(v^0,v^0)^{1/2}  \\
& \le  \overline{C}^2 \sum_{E\in\T_h} \|(I-\Pinabla)v^I\|_{1,E} \|v^0\|_{1,E} .\\
\endaligned
\end{equation*}

We estimate the term $\|(I-\Pinabla)v^I\|_{1,E}$, taking into account that
$\Pinabla$ is the elliptic projection onto $\P^k(E)$,
see~\eqref{eq:pinabla}. Therefore, for $v\in H^{1+s}(E)$, it holds that $\|v-\Pinabla v\|_{1,E}\le C h^s \|v\|_{1+s,E}$. 
Hence, using Lemma~\ref{le:interp}, we can write
$$
\aligned
\|(I-\Pinabla)v^I\|_{1,E} & \le \|v^I-v\|_{1,E} + \|v-\Pinabla v\|_{1,E} +\| \Pinabla(v-v^I)\|_{1,E} \\
& \le Ch^{s}\|v\|_{1+s,E}.
\endaligned
$$
Summing on $E\in\T_h$, we obtain that
$$
\|v^0\|_1\le Ch^s \|v\|_{1+s},
$$
which, together with~\eqref{eq:triangular}, implies property $\mathbf{P2}$.

\end{proof}

\section{Numerical results}
\label{se:num}
This section is devoted to the confirmation of  the theoretical results
obtained previously and to numerically investigate whether property
$\mathbf{P2}$ holds also for virtual element approximation of order higher
than $k=2$.
For the computations, we use the Matlab code developed by the team
at the Department of Mathematics and Applications (University of
Milano-Bicocca), within the ERC consolidator grant CAVE (Challenges and
Advancement in Virtual Elements)~\cite{CAVE}.
We remark that the efficient resolution of the algebraic eigenvalue problem is
not a scope of the present paper.

Let $\Omega$ be the unit square, and let us consider the Laplace eigenvalue problem: 
$$
\aligned
& -\Delta u = \lambda u  \quad && \text{in } \Omega \\
& u = 0 && \text{on } \partial\Omega.
\endaligned
$$
In this case, the exact eigensolutions are well-know and, for positive $i,j\in\mathbb{N}$, are given by
$$
\lambda_{ij} = (i^2 + j^2)\pi^2 , \quad u_{ij} = \sin(i\pi x )\sin(j\pi y).
$$

For readability reasons, in the tables reporting the numerical results, we shall display the value of the computed eigenvalues divided by $\pi^2$. 

\begin{figure}[h]
\begin{center}
\subfigure[$\mathcal{T}$]
{
\includegraphics[width=0.22\textwidth]{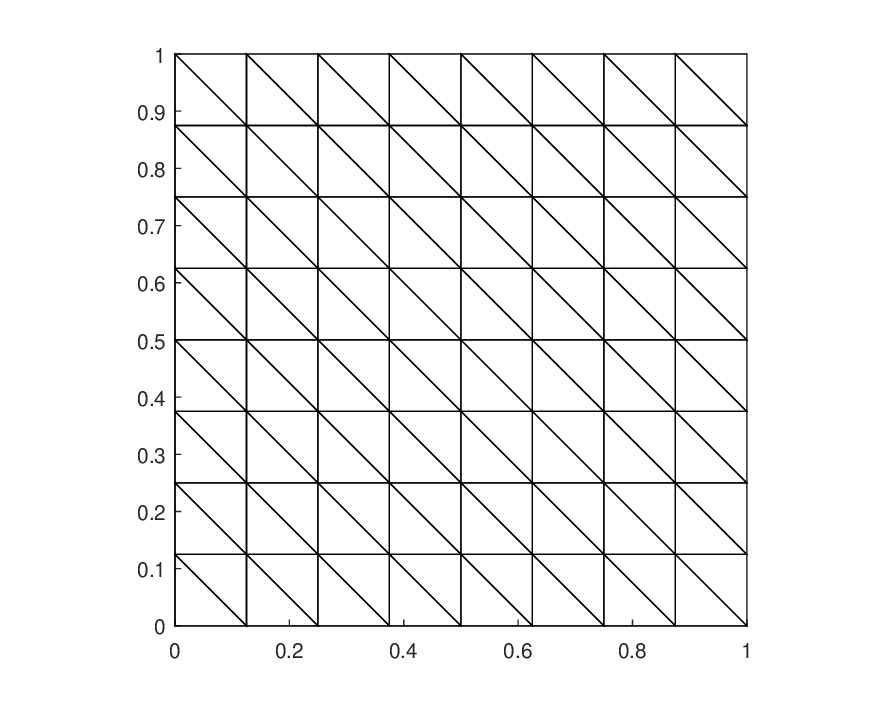}   
}
\subfigure[$\mathcal{S}$]
{
\includegraphics[width=0.22\textwidth]{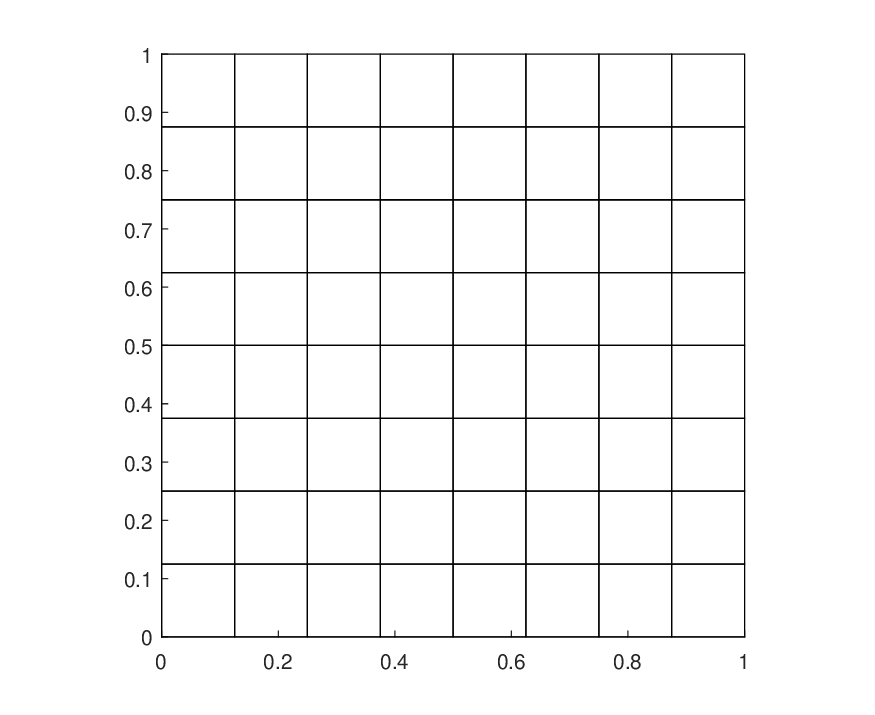} 
}
\subfigure[$\mathcal{V}$]
{
\includegraphics[width=0.22\textwidth]{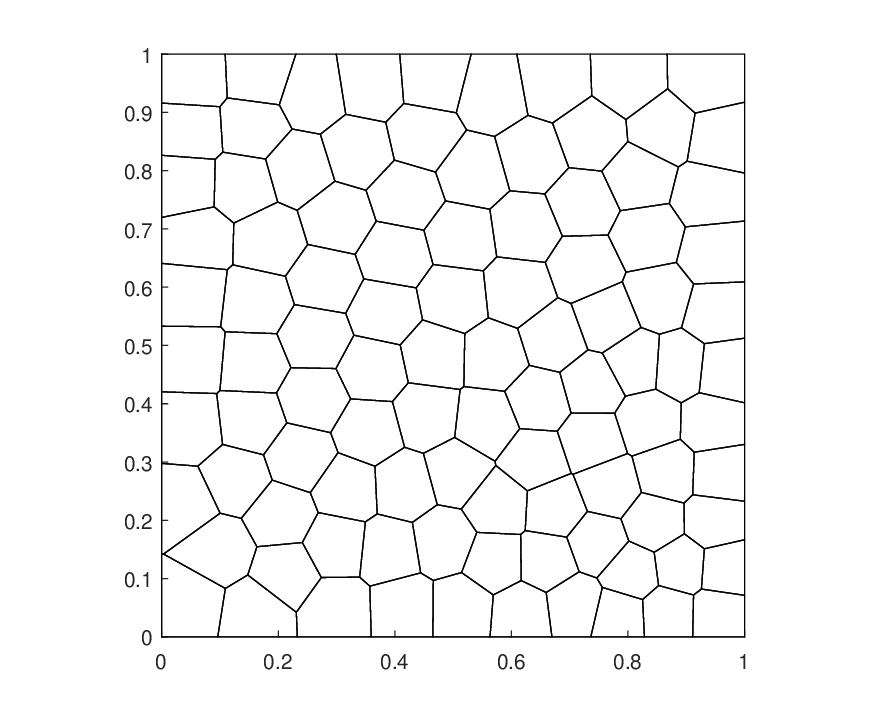}
}
\subfigure[$\mathcal{H}$]
{
\includegraphics[width=0.24\textwidth]{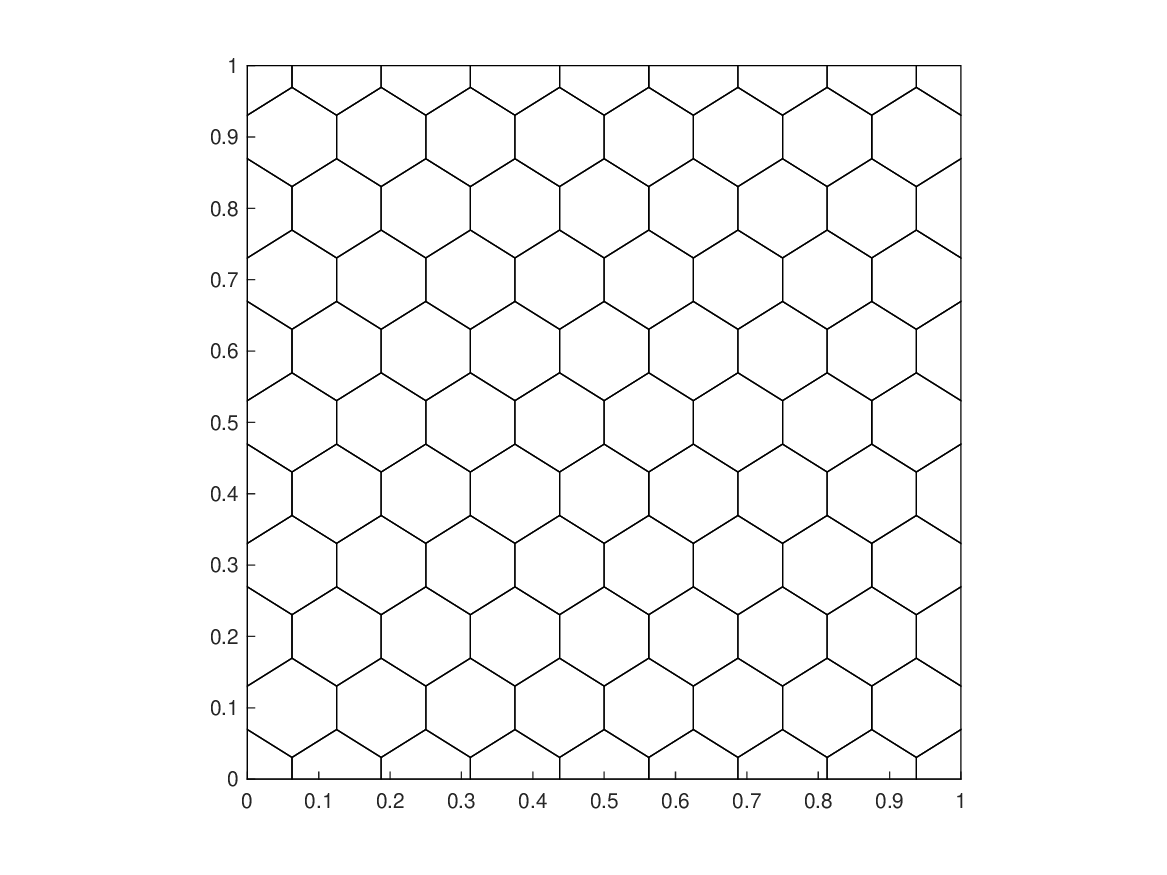} 
}
\end{center}
\caption{Coarsest meshes of different types:  $\mathcal{T}$ triangles,
$\mathcal{S}$ squares, $\mathcal{V}$ Voronoi, $\mathcal{H}$ hexagons}
\label{fg:mesh}
\end{figure}

In the virtual element discretization we employ four different mesh types,
labeled by $\mathcal{T}, \mathcal{S}, \mathcal{V}, \mathcal{H}$ and reported in
Figure~\ref{fg:mesh}. In particular, meshes $\mathcal{T}$ and $\mathcal{S}$ are
standard uniform meshes made of triangles and squares obtained subdividing each
edge of the square into $N$ parts. For the triangular mesh, the squares
are further subdivided into two triangles.
The mesh $ \mathcal{V}$ is the so-called Voronoi mesh made of $P$ polygons with
possibly different number of edges. Finally, the mesh $\mathcal{H}$ is obtained 
by using mainly hexagons except for boundary elements. In this case the
refinement level of the mesh is indicated by \texttt{n\,x\,m} where \texttt{n}
and \texttt{m} stand for the number of hexagons along the horizontal and the
vertical direction, respectively. The refinement level of the meshes will be
denoted by $N$, $N$, $P$, and \texttt{m}, respectively. The rates of
convergence are computed with respect to the maximum diameter of the
elements $h$.

The algebraic eigenvalue problem associated with the virtual element
approximation of our model problem (see~\eqref{eq:Poissonh}) is
\begin{equation}
\label{eq:matrixform}
\m{A}\m{x}=\m{\lambda}\m{Bx},
\end{equation}
where $\m{A}$ and $\m{B}$ are the matrices associated with the
bilinear forms $a_h$ and $b_h$, respectively. The matrix $\m{A}$ is stabilized
as explained in~\eqref{eq:aEh}, while $\m{B}$ is not stabilized.
In the numerical tests, the local stabilized bilinear
form $S^E$ is taken as the so-called \emph{dofi-dofi}, obtained using the
vector of the degrees of freedom, that is
\[
S^E(\uh,\vh)=\alpha\sum_{i=1}^{N_E}\m{u}_i\m{v}_i,
\]
where $N_E$ is the dimension of the local VEM space and $\m{u}_i$ are
components of the vector containing the degrees of freedom of $\uh$
and $\alpha\in\RE$ is a positive number. In our numerical experiments, we
fix $\alpha=1$. 


We aim at identifying the kernel $\Kb$ and at checking the convergence of the
computed eigenvalues in order to confirm the theoretical findings. 

First of all, we evaluate the dimension of the kernel of $\m{B}$ by computing
the rank of the matrix. This can be easily done in Matlab for symmetric and
positive semidefinite matrices in sparse form with the command \texttt{chol}.
The dimension of the kernel $\Kb$ for different types of mesh is reported in
the corresponding tables organized as follows. 
The columns contain the value of the dimension of the kernels 
with respect to the degree of the polynomial used in the definition of the
virtual element space. Within
parenthesis, we display the dimension of the matrix. The rows 
refer to the level of refinement of the mesh.
For each type of mesh and  degree of the VEM spaces, we report also the first ten eigenvalues and the corresponding rates of convergence 
in the case of uniform refinement of the mesh. 

\subsection{Triangular and square meshes}
\label{sec:ts}

\begin{table}
\begin{center}
\caption{Dimension of the kernel of matrix $\m{B}$ for triangular mesh
$\mathcal{T}$}
\label{tb:kernelT}
\begin{tabular}{r|l r| l r| l r| l r }
$\mathcal{T}$ &\multicolumn{6}{c}{ $\dim(\Kb)$ ($\dim(V_h)$)}\\[2pt]
\hline
N& \multicolumn{2}{c}{$k=1$} & \multicolumn{2}{c}{$k=2$} &
\multicolumn{2}{c}{$k=3$}& \multicolumn{2}{c}{$k=4$}\\[2pt]
\hline
 4 & 0 &  (9)    & 0 &  (81)    & 0 &  (185)    & 0 & (321)\\
 8 & 0 &  (49)   & 0 &  (353)   & 0 &  (785)    & 0 & (1345)\\
16 & 0 &  (225)  & 0 &  (1473)  & 0 &  (3233)   & 0 & (5505)\\
32 & 0 &  (961)  & 0 &  (6017)  & 0 &  (13121)  & 0 & (22273)\\
64 & 0 &  (3969) & 0 &  (24321) & 0 &  (52865)  & 0& (89601)\\[2pt]
\hline
\end{tabular}
\end{center}
\end{table}
\begin{table}[h]
\begin{center}
\caption{Dimension of the kernel of matrix $\m{B}$ for square mesh
$\mathcal{S}$}
\label{tb:kernelS}
\begin{tabular}{r|l r| l r| l r| r r}
$\mathcal{S}$ &\multicolumn{6}{c}{ $\dim(\Kb)$ ($\dim(V_h)$)}\\[2pt]
\hline
N& \multicolumn{2}{c}{$k=1$} & \multicolumn{2}{c}{$k=2$} &
\multicolumn{2}{c}{$k=3$}& \multicolumn{2}{c}{$k=4$}\\[2pt] 
\hline
 4 & 0 &  (9)    & 0 &  (49)    & 0 &  (105)    & 0 & (177)\\
 8 & 0 &  (49)   & 0 &  (225)   & 0 &  (465)    & 0 & (769)\\
16 & 0 &  (225)  & 0 &  (961)   & 0 &  (1953)   & 0 & (3201)\\
32 & 0 &  (961)  & 0 &  (3969)  & 0 &  (8001)   & 0 & (13057)\\
64 & 0 &  (3969) & 0 &  (16129) & 0 &  (32385)  & 43 & (52737)\\[2pt]
\hline
\end{tabular}
\end{center}
\end{table}

All cases of triangular meshes reported in Table~\ref{tb:kernelT} have
positive definite matrices $\m{B}$, while the last case in
Table~\ref{tb:kernelS} (mesh of squares for $N=64$ and $k=4$) shows a
nontrivial kernel. Errors and rates of convergence for the first 10
eigenvalues with $k=1,\dots,4$ are reported in
Tables~\ref{tb:tk1}-\ref{tb:tk4} for triangular meshes, and in
Tables~\ref{tb:sk1}-\ref{tb:sk4} for square meshes, confirming the correct
spectral approximation and the optimal order of convergence.

\begin{table}
\caption{ First 10 eigenvalues on $\mathcal{T}$ with $k=1$ }
\label{tb:tk1}
\begin{tabular}{r|rrrrr}
\hline
Exact&\multicolumn{5}{c}{Errors (rate)}\\[3pt]
\hline
   2 & 3.2e-01 & 7.8e-02 (2.03) & 1.9e-02 (2.01) & 4.8e-03 (2.01) & 1.2e-03 (2.00) \\
   5 & 1.3e+00 & 3.3e-01 (2.01) & 8.3e-02 (2.00) & 2.1e-02 (2.00) & 5.2e-03 (1.99) \\
   5 & 2.3e+00 & 5.3e-01 (2.08) & 1.3e-01 (2.03) & 3.2e-02 (2.01) & 8.1e-03 (2.00) \\
   8 & 4.2e+00 & 1.2e+00 (1.83) & 3.1e-01 (1.95) & 7.7e-02 (1.99) & 1.9e-02 (1.99) \\
  10 & 5.6e+00 & 1.5e+00 (1.84) & 3.8e-01 (2.02) & 9.5e-02 (2.01) & 2.4e-02 (2.00) \\
  10 & 6.8e+00 & 1.7e+00 (2.00) & 3.9e-01 (2.11) & 9.5e-02 (2.03) & 2.4e-02 (2.01) \\
  13 & 7.9e+00 & 2.2e+00 (1.83) & 5.7e-01 (1.96) & 1.4e-01 (1.99) & 3.6e-02 (2.00) \\
  13 & 1.3e+01 & 4.0e+00 (1.71) & 9.8e-01 (2.03) & 2.4e-01 (2.01) & 6.1e-02 (2.00) \\
  17 & 1.5e+01 & 4.3e+00 (1.83) & 1.0e+00 (2.06) & 2.6e-01 (2.02) & 6.4e-02 (2.01) \\
  17 & 8.3e+01 & 4.6e+00 (4.18) & 1.1e+00 (2.10) & 2.6e-01 (2.03) & 6.5e-02 (2.01) \\
\hline
$N$ &\multicolumn{1}{c}{4} &\multicolumn{1}{c}{8}&\multicolumn{1}{c}{16} &\multicolumn{1}{c}{32} &\multicolumn{1}{c}{64} \\
\hline
\end{tabular}
\end{table}
\begin{table}
\caption{ First 10 eigenvalues on $\mathcal{T}$ with $k=1$ }
\label{tb:tk2}
\begin{tabular}{r|rrrrr}
\hline
Exact&\multicolumn{5}{c}{Errors (rate)}\\[3pt]
\hline
   2 & 3.2e-01 & 7.8e-02 (2.03) & 1.9e-02 (2.01) & 4.8e-03 (2.01) & 1.2e-03 (2.00) \\
   5 & 1.3e+00 & 3.3e-01 (2.01) & 8.3e-02 (2.00) & 2.1e-02 (2.00) & 5.2e-03 (1.99) \\
   5 & 2.3e+00 & 5.3e-01 (2.08) & 1.3e-01 (2.03) & 3.2e-02 (2.01) & 8.1e-03 (2.00) \\
   8 & 4.2e+00 & 1.2e+00 (1.83) & 3.1e-01 (1.95) & 7.7e-02 (1.99) & 1.9e-02 (1.99) \\
  10 & 5.6e+00 & 1.5e+00 (1.84) & 3.8e-01 (2.02) & 9.5e-02 (2.01) & 2.4e-02 (2.00) \\
  10 & 6.8e+00 & 1.7e+00 (2.00) & 3.9e-01 (2.11) & 9.5e-02 (2.03) & 2.4e-02 (2.01) \\
  13 & 7.9e+00 & 2.2e+00 (1.83) & 5.7e-01 (1.96) & 1.4e-01 (1.99) & 3.6e-02 (2.00) \\
  13 & 1.3e+01 & 4.0e+00 (1.71) & 9.8e-01 (2.03) & 2.4e-01 (2.01) & 6.1e-02 (2.00) \\
  17 & 1.5e+01 & 4.3e+00 (1.83) & 1.0e+00 (2.06) & 2.6e-01 (2.02) & 6.4e-02 (2.01) \\
  17 & 8.3e+01 & 4.6e+00 (4.18) & 1.1e+00 (2.10) & 2.6e-01 (2.03) & 6.5e-02 (2.01) \\
\hline
$N$ &\multicolumn{1}{c}{4} &\multicolumn{1}{c}{8}&\multicolumn{1}{c}{16} &\multicolumn{1}{c}{32} &\multicolumn{1}{c}{64} \\
\hline
\end{tabular}
\end{table}
\begin{table}
\caption{ First 10 eigenvalues on $\mathcal{T}$ with $k=3$ }
\label{tb:tk3}
\begin{tabular}{r|rrrrr}
\hline
Exact&\multicolumn{5}{c}{Errors (rate)}\\[3pt]
\hline
   2 & 6.5e-05 & 1.0e-06 (6.01) & 1.6e-08 (6.02) & 2.4e-10 (6.01) & 3.9e-12 (5.94) \\
   5 & 1.1e-03 & 1.8e-05 (5.93) & 2.8e-07 (6.02) & 4.3e-09 (6.02) & 6.8e-11 (5.99) \\
   5 & 2.2e-03 & 3.6e-05 (5.93) & 5.5e-07 (6.01) & 8.5e-09 (6.01) & 1.3e-10 (6.03) \\
   8 & 1.4e-02 & 2.6e-04 (5.75) & 4.0e-06 (6.00) & 6.2e-08 (6.01) & 9.6e-10 (6.01) \\
  10 & 1.6e-02 & 2.9e-04 (5.79) & 4.6e-06 (6.00) & 7.1e-08 (6.01) & 1.1e-09 (6.01) \\
  10 & 1.6e-02 & 3.0e-04 (5.79) & 4.6e-06 (6.01) & 7.1e-08 (6.01) & 1.1e-09 (6.01) \\
  13 & 4.4e-02 & 9.4e-04 (5.56) & 1.5e-05 (5.97) & 2.3e-07 (6.02) & 3.6e-09 (6.01) \\
  13 & 9.1e-02 & 2.1e-03 (5.46) & 3.3e-05 (5.97) & 5.1e-07 (6.01) & 8.0e-09 (6.01) \\
  17 & 1.4e-01 & 1.9e-03 (6.18) & 3.1e-05 (5.98) & 4.7e-07 (6.01) & 7.4e-09 (6.01) \\
  17 & 1.4e-01 & 2.0e-03 (6.13) & 3.2e-05 (5.98) & 4.9e-07 (6.01) & 7.7e-09 (6.00) \\
\hline
$N$ &\multicolumn{1}{c}{4} &\multicolumn{1}{c}{8}&\multicolumn{1}{c}{16} &\multicolumn{1}{c}{32} &\multicolumn{1}{c}{64} \\
\hline
\end{tabular}
\end{table}
\begin{table}
\caption{ First 10 eigenvalues on $\mathcal{T}$ with $k=4$ }
\label{tb:tk4}
\begin{tabular}{r|rrrrr}
\hline
Exact&\multicolumn{5}{c}{Errors (rate)}\\[3pt]
\hline
   2 & 4.9e-07 & 2.0e-09 (7.94) & 9.5e-12 (7.71) & 3.2e-12 (1.55) & 9.9e-12 (-1.61) \\
   5 & 1.4e-05 & 6.2e-08 (7.77) & 2.5e-10 (7.94) & 2.2e-12 (6.87) & 6.6e-12 (-1.60) \\
   5 & 3.8e-05 & 1.6e-07 (7.90) & 6.3e-10 (7.97) & 3.7e-12 (7.42) & 1.3e-11 (-1.78) \\
   8 & 4.1e-04 & 1.9e-06 (7.75) & 7.9e-09 (7.92) & 3.2e-11 (7.93) & 1.7e-12 (4.23) \\
  10 & 4.2e-04 & 1.9e-06 (7.76) & 7.8e-09 (7.94) & 3.0e-11 (8.04) & 5.9e-12 (2.34) \\
  10 & 4.2e-04 & 1.9e-06 (7.77) & 7.8e-09 (7.94) & 3.1e-11 (8.00) & 9.3e-12 (1.72) \\
  13 & 1.7e-03 & 9.8e-06 (7.46) & 4.2e-08 (7.87) & 1.7e-10 (7.95) & 3.1e-12 (5.77) \\
  13 & 5.0e-03 & 2.4e-05 (7.70) & 1.0e-07 (7.89) & 4.1e-10 (7.97) & 3.2e-13 (10.33) \\
  17 & 2.2e-03 & 1.7e-05 (7.01) & 7.2e-08 (7.88) & 2.9e-10 (7.98) & 2.1e-12 (7.08) \\
  17 & 3.0e-03 & 1.9e-05 (7.30) & 7.9e-08 (7.90) & 3.1e-10 (7.99) & 1.0e-11 (4.91) \\
\hline
$N$ &\multicolumn{1}{c}{4} &\multicolumn{1}{c}{8}&\multicolumn{1}{c}{16} &\multicolumn{1}{c}{32} &\multicolumn{1}{c}{64} \\
\hline
\end{tabular}
\end{table}
\begin{table}
\caption{ First 10 eigenvalues on $\mathcal{S}$ with $k=1$ }
\label{tb:sk1}
\begin{tabular}{r|rrrrr}
\hline
Exact&\multicolumn{5}{c}{Errors (rate)}\\[3pt]
\hline
   2 & 1.7e-01 & 3.9e-02 (2.09) & 9.7e-03 (2.02) & 2.4e-03 (2.01) & 6.0e-04 (2.00) \\
   5 & 1.3e+00 & 2.8e-01 (2.17) & 6.8e-02 (2.04) & 1.7e-02 (2.01) & 4.2e-03 (2.01) \\
   5 & 1.3e+00 & 2.8e-01 (2.17) & 6.8e-02 (2.04) & 1.7e-02 (2.01) & 4.2e-03 (2.01) \\
   8 & 3.7e+00 & 6.7e-01 (2.45) & 1.6e-01 (2.09) & 3.9e-02 (2.02) & 9.7e-03 (2.00) \\
  10 & 5.1e+00 & 1.2e+00 (2.04) & 3.0e-01 (2.06) & 7.3e-02 (2.02) & 1.8e-02 (2.00) \\
  10 & 5.1e+00 & 1.2e+00 (2.04) & 3.0e-01 (2.06) & 7.3e-02 (2.02) & 1.8e-02 (2.00) \\
  13 & 1.2e+01 & 1.9e+00 (2.64) & 4.4e-01 (2.14) & 1.1e-01 (2.03) & 2.7e-02 (2.01) \\
  13 & 1.2e+01 & 1.9e+00 (2.64) & 4.4e-01 (2.14) & 1.1e-01 (2.03) & 2.7e-02 (2.01) \\
  17 & 4.4e+01 & 3.8e+00 (3.52) & 9.0e-01 (2.08) & 2.2e-01 (2.03) & 5.5e-02 (2.01) \\
  17 & 8.3e+01 & 3.8e+00 (4.45) & 9.0e-01 (2.08) & 2.2e-01 (2.03) & 5.5e-02 (2.01) \\
\hline
$N$ &\multicolumn{1}{c}{4} &\multicolumn{1}{c}{8}&\multicolumn{1}{c}{16} &\multicolumn{1}{c}{32} &\multicolumn{1}{c}{64} \\
\hline
\end{tabular}
\end{table}
\begin{table}
\caption{ First 10 eigenvalues on $\mathcal{S}$ with $k=2$ }
\label{tb:sk2}
\begin{tabular}{r|rrrrr}
\hline
Exact&\multicolumn{5}{c}{Errors (rate)}\\[3pt]
\hline
   2 & 7.6e-04 & 4.3e-05 (4.14) & 2.6e-06 (4.04) & 1.6e-07 (4.01) & 1.0e-08 (4.00) \\
   5 & 3.0e-02 & 1.9e-03 (3.97) & 1.2e-04 (3.99) & 7.4e-06 (4.00) & 4.6e-07 (4.00) \\
   5 & 3.0e-02 & 1.9e-03 (3.97) & 1.2e-04 (3.99) & 7.4e-06 (4.00) & 4.6e-07 (4.00) \\
   8 & 6.3e-02 & 3.0e-03 (4.38) & 1.7e-04 (4.14) & 1.0e-05 (4.04) & 6.5e-07 (4.01) \\
  10 & 3.1e-01 & 2.2e-02 (3.81) & 1.4e-03 (3.95) & 9.0e-05 (3.99) & 5.6e-06 (4.00) \\
  10 & 3.1e-01 & 2.2e-02 (3.81) & 1.4e-03 (3.95) & 9.0e-05 (3.99) & 5.6e-06 (4.00) \\
  13 & 3.9e-01 & 2.2e-02 (4.16) & 1.3e-03 (4.07) & 8.0e-05 (4.02) & 5.0e-06 (4.00) \\
  13 & 3.9e-01 & 2.2e-02 (4.16) & 1.3e-03 (4.07) & 8.0e-05 (4.02) & 5.0e-06 (4.00) \\
  17 & 3.0e-01 & 1.2e-01 (1.31) & 8.0e-03 (3.89) & 5.1e-04 (3.97) & 3.2e-05 (3.99) \\
  17 & 3.0e-01 & 1.2e-01 (1.31) & 8.0e-03 (3.89) & 5.1e-04 (3.97) & 3.2e-05 (3.99) \\
\hline
$N$ &\multicolumn{1}{c}{4} &\multicolumn{1}{c}{8}&\multicolumn{1}{c}{16} &\multicolumn{1}{c}{32} &\multicolumn{1}{c}{64} \\
\hline
\end{tabular}
\end{table}
\begin{table}
\caption{ First 10 eigenvalues on square mesh $\mathcal{S}$ with $k=3$ }
\label{tb:sk3}
\begin{tabular}{r|rrrrr}
\hline
Exact&\multicolumn{5}{c}{Errors (rate)}\\[3pt]
\hline
   2 & 6.5e-05 & 1.1e-06 (5.91) & 1.7e-08 (5.98) & 2.7e-10 (6.00) & 3.0e-12 (6.51) \\
   5 & 1.4e-03 & 2.8e-05 (5.62) & 4.6e-07 (5.91) & 7.3e-09 (5.98) & 1.1e-10 (6.02) \\
   5 & 1.4e-03 & 2.8e-05 (5.62) & 4.6e-07 (5.91) & 7.3e-09 (5.98) & 1.1e-10 (6.01) \\
   8 & 1.3e-02 & 2.6e-04 (5.64) & 4.4e-06 (5.91) & 6.9e-08 (5.98) & 1.1e-09 (6.00) \\
  10 & 1.2e-02 & 3.3e-04 (5.21) & 5.8e-06 (5.82) & 9.4e-08 (5.96) & 1.5e-09 (5.99) \\
  10 & 1.2e-02 & 3.3e-04 (5.21) & 5.8e-06 (5.82) & 9.4e-08 (5.96) & 1.5e-09 (5.99) \\
  13 & 5.3e-02 & 1.4e-03 (5.23) & 2.5e-05 (5.81) & 4.1e-07 (5.95) & 6.4e-09 (5.99) \\
  13 & 5.3e-02 & 1.4e-03 (5.23) & 2.5e-05 (5.81) & 4.1e-07 (5.95) & 6.4e-09 (5.99) \\
  17 & 2.1e-01 & 2.4e-03 (6.45) & 4.5e-05 (5.74) & 7.4e-07 (5.94) & 1.2e-08 (5.99) \\
  17 & 2.1e-01 & 2.4e-03 (6.45) & 4.5e-05 (5.74) & 7.4e-07 (5.94) & 1.2e-08 (5.99) \\
\hline
$N$ &\multicolumn{1}{c}{4} &\multicolumn{1}{c}{8}&\multicolumn{1}{c}{16} &\multicolumn{1}{c}{32} &\multicolumn{1}{c}{64} \\
\hline
\end{tabular}
\end{table}
\begin{table}
\caption{ First 10 eigenvalues on square mesh $\mathcal{S}$ with $k=4$ }
\label{tb:sk4}
\begin{tabular}{r|rrrrr}
\hline
Exact&\multicolumn{5}{c}{Errors (rate)}\\[3pt]
\hline
   2 & 5.2e-07 & 2.3e-09 (7.83) & 8.8e-12 (8.02) & 7.4e-13 (3.57) & 1.5e-11 (-4.37) \\
   5 & 3.5e-05 & 1.7e-07 (7.69) & 7.0e-10 (7.92) & 1.9e-12 (8.55) & 3.8e-12 (-1.02) \\
   5 & 3.5e-05 & 1.7e-07 (7.69) & 7.0e-10 (7.92) & 5.0e-12 (7.13) & 2.1e-11 (-2.09) \\
   8 & 3.1e-04 & 2.1e-06 (7.23) & 9.1e-09 (7.83) & 3.7e-11 (7.94) & 8.5e-12 (2.12) \\
  10 & 7.9e-04 & 4.0e-06 (7.63) & 1.7e-08 (7.90) & 6.7e-11 (7.97) & 9.1e-12 (2.87) \\
  10 & 7.9e-04 & 4.0e-06 (7.63) & 1.7e-08 (7.90) & 6.7e-11 (7.96) & 1.3e-11 (2.34) \\
  13 & 2.3e-03 & 2.0e-05 (6.84) & 9.3e-08 (7.75) & 3.8e-10 (7.94) & 3.2e-12 (6.89) \\
  13 & 2.3e-03 & 2.0e-05 (6.84) & 9.3e-08 (7.75) & 3.8e-10 (7.94) & 1.3e-11 (4.85) \\
  17 & 1.0e-02 & 4.5e-05 (7.81) & 1.9e-07 (7.88) & 7.7e-10 (7.97) & 1.1e-11 (6.14) \\
  17 & 1.0e-02 & 4.5e-05 (7.81) & 1.9e-07 (7.88) & 7.7e-10 (7.97) & 1.7e-11 (5.53) \\
\hline
$N$ &\multicolumn{1}{c}{4} &\multicolumn{1}{c}{8}&\multicolumn{1}{c}{16} &\multicolumn{1}{c}{32} &\multicolumn{1}{c}{64} \\
\hline
\end{tabular}
\end{table}

\subsection{Voronoi meshes}
\label{sec:voronoi}

\begin{table}
\begin{center}
\caption{Dimension of the kernel of matrix $\m{B}$ for Voronoi mesh
$\mathcal{V}$}
\label{tb:kernelV}
\begin{tabular}{r|r r| r r| r r| r r}
$\mathcal{V}$ &\multicolumn{6}{c}{ $\dim(\Kb)$ ($\dim(V_h)$)}\\[2pt]
\hline
$P$& \multicolumn{2}{c}{$k=1$} & \multicolumn{2}{c}{$k=2$} &
\multicolumn{2}{c}{$k=3$}& \multicolumn{2}{c}{$k=4$}\\[2pt]
\hline
 50 & 0 &  (73)   & 0 &  (245)  & 0    & (467)   & 5    & (739)\\
100 & 0 &  (163)  & 0 &  (525)  & 1    & (987)   & 57   & (1549)\\
200 & 0 &  (346)  & 0 &  (1091  & 43   & (2036)  & 188  & (3181)\\
400 & 0 &  (727)  & 0 &  (2253) & 182  & (4179)  & 512  & (6505)\\
800 & 0 &  (1500) & 0 &  (4599) & 504  & (8498)  & 1212 & (13197)\\[2pt]
\hline
\end{tabular}
\end{center}
\end{table}

For Voronoi mesh, the rank of the matrix $\m{B}$ is reported
in Table~\ref{tb:kernelV}. In this case, we have that for $k=1$ and $k=2$ the
kernel is again reduced to $\{0\}$, while for $k>2$ this is not true anymore.
Although the cases $k>2$ are not covered by our theory, from
Tables~\ref{tb:vk1}-\ref{tb:vk4} we see that the rate of convergence is still
optimal and no spurious eigenvalues appear in the spectrum.
\begin{table}
\caption{ First 10 eigenvalues on $\mathcal{V}$ with $k=1$ }
\label{tb:vk1}
\begin{tabular}{r|rrrrr}
\hline
Exact&\multicolumn{5}{c}{Errors (rate)}\\[3pt]
\hline
   2 & 5.2e-02 & 2.3e-02 (2.27) & 1.0e-02 (3.12) & 5.4e-03 (2.36) & 2.5e-03 (2.02) \\
   5 & 2.4e-01 & 1.3e-01 (1.69) & 6.3e-02 (2.96) & 3.1e-02 (2.52) & 1.6e-02 (1.76) \\
   5 & 3.4e-01 & 1.5e-01 (2.37) & 6.8e-02 (3.02) & 3.2e-02 (2.64) & 1.6e-02 (1.84) \\
   8 & 7.6e-01 & 3.7e-01 (2.01) & 1.7e-01 (2.97) & 8.6e-02 (2.51) & 4.2e-02 (1.86) \\
  10 & 1.0e+00 & 5.3e-01 (1.83) & 2.5e-01 (2.91) & 1.3e-01 (2.41) & 6.1e-02 (1.96) \\
  10 & 1.3e+00 & 5.6e-01 (2.41) & 2.6e-01 (2.99) & 1.3e-01 (2.46) & 6.4e-02 (1.88) \\
  13 & 1.8e+00 & 9.3e-01 (1.80) & 4.4e-01 (2.93) & 2.2e-01 (2.46) & 1.1e-01 (1.84) \\
  13 & 2.3e+00 & 1.0e+00 (2.30) & 4.5e-01 (3.12) & 2.3e-01 (2.36) & 1.1e-01 (1.95) \\
  17 & 2.8e+00 & 1.6e+00 (1.69) & 7.5e-01 (2.89) & 3.6e-01 (2.55) & 1.8e-01 (1.82) \\
  17 & 3.9e+00 & 1.6e+00 (2.55) & 7.6e-01 (2.85) & 3.7e-01 (2.53) & 1.8e-01 (1.85) \\
\hline
$P$ &\multicolumn{1}{c}{50} &\multicolumn{1}{c}{100}&\multicolumn{1}{c}{200} &\multicolumn{1}{c}{400} &\multicolumn{1}{c}{800} \\
\hline
\end{tabular}
\end{table}

\begin{table}
\caption{ First 10 eigenvalues on $\mathcal{V}$ with $k=2$ }
\label{tb:vk2}
\begin{tabular}{r|rrrrr}
\hline
Exact&\multicolumn{5}{c}{Errors (rate)}\\[3pt]
\hline
   2 & 2.4e-04 & 6.8e-05 (3.53) & 1.9e-05 (5.02) & 4.7e-06 (4.99) & 1.3e-06 (3.43) \\
   5 & 4.1e-03 & 1.1e-03 (3.57) & 3.2e-04 (5.04) & 7.5e-05 (5.10) & 1.9e-05 (3.57) \\
   5 & 5.7e-03 & 1.4e-03 (3.91) & 3.3e-04 (5.74) & 8.2e-05 (4.91) & 2.0e-05 (3.72) \\
   8 & 1.5e-02 & 4.3e-03 (3.41) & 1.1e-03 (5.25) & 3.1e-04 (4.60) & 7.9e-05 (3.58) \\
  10 & 3.5e-02 & 1.1e-02 (3.26) & 2.5e-03 (5.77) & 6.1e-04 (4.98) & 1.6e-04 (3.58) \\
  10 & 5.0e-02 & 1.1e-02 (4.13) & 2.7e-03 (5.63) & 6.7e-04 (4.99) & 1.6e-04 (3.75) \\
  13 & 5.8e-02 & 1.7e-02 (3.44) & 4.9e-03 (4.84) & 1.3e-03 (4.76) & 3.4e-04 (3.51) \\
  13 & 8.2e-02 & 2.1e-02 (3.79) & 5.0e-03 (5.64) & 1.4e-03 (4.44) & 3.5e-04 (3.70) \\
  17 & 1.7e-01 & 5.2e-02 (3.31) & 1.2e-02 (5.76) & 3.1e-03 (4.82) & 7.6e-04 (3.68) \\
  17 & 2.3e-01 & 5.4e-02 (4.06) & 1.3e-02 (5.56) & 3.2e-03 (5.06) & 7.9e-04 (3.64) \\
\hline
$P$ &\multicolumn{1}{c}{50} &\multicolumn{1}{c}{100}&\multicolumn{1}{c}{200} &\multicolumn{1}{c}{400} &\multicolumn{1}{c}{800} \\
\hline
\end{tabular}
\end{table}

\begin{table}
\caption{ First 10 eigenvalues on $\mathcal{V}$ with $k=3$ }
\label{tb:vk3}
\begin{tabular}{r|rrrrr}
\hline
Exact&\multicolumn{5}{c}{Errors (rate)}\\[3pt]
\hline
   2 & 2.7e-06 & 2.6e-07 (6.56) & 3.0e-08 (8.51) & 4.5e-09 (6.73) & 5.4e-10 (5.55) \\
   5 & 6.2e-05 & 8.7e-06 (5.45) & 1.1e-06 (8.27) & 1.3e-07 (7.47) & 1.7e-08 (5.31) \\
   5 & 8.9e-05 & 9.7e-06 (6.17) & 1.1e-06 (8.54) & 1.5e-07 (7.07) & 1.8e-08 (5.57) \\
   8 & 5.3e-04 & 7.3e-05 (5.52) & 8.1e-06 (8.69) & 1.0e-06 (7.41) & 1.3e-07 (5.39) \\
  10 & 9.0e-04 & 1.3e-04 (5.43) & 1.5e-05 (8.39) & 2.0e-06 (7.27) & 2.4e-07 (5.53) \\
  10 & 1.2e-03 & 1.4e-04 (6.10) & 1.6e-05 (8.43) & 2.2e-06 (7.05) & 2.7e-07 (5.54) \\
  13 & 2.8e-03 & 4.0e-04 (5.43) & 5.0e-05 (8.21) & 6.2e-06 (7.39) & 8.3e-07 (5.28) \\
  13 & 3.8e-03 & 5.1e-04 (5.56) & 5.1e-05 (9.09) & 7.1e-06 (7.05) & 8.5e-07 (5.54) \\
  17 & 6.3e-03 & 9.9e-04 (5.12) & 1.2e-04 (8.23) & 1.7e-05 (7.17) & 2.0e-06 (5.49) \\
  17 & 9.4e-03 & 1.0e-03 (6.13) & 1.3e-04 (8.24) & 1.7e-05 (7.17) & 2.1e-06 (5.50) \\
\hline
$P$ &\multicolumn{1}{c}{50} &\multicolumn{1}{c}{100}&\multicolumn{1}{c}{200} &\multicolumn{1}{c}{400} &\multicolumn{1}{c}{800} \\
\hline
\end{tabular}
\end{table}

\begin{table}
\caption{ First 10 eigenvalues on $\mathcal{V}$ with $k=4$ }
\label{tb:vk4}
\begin{tabular}{r|rrrrr}
\hline
Exact&\multicolumn{5}{c}{Errors (rate)}\\[3pt]
\hline
   2 & 6.6e-09 & 2.7e-10 (8.93) & 1.5e-11 (11.34) & 1.2e-12 (9.15) & 1.2e-13 (5.89) \\
   5 & 1.9e-07 & 1.4e-08 (7.23) & 5.1e-10 (13.14) & 6.4e-11 (7.38) & 2.7e-12 (8.30) \\
   5 & 6.8e-07 & 1.7e-08 (10.16) & 1.0e-09 (11.19) & 7.0e-11 (9.54) & 4.1e-12 (7.41) \\
   8 & 4.7e-06 & 2.1e-07 (8.67) & 8.7e-09 (12.52) & 7.7e-10 (8.62) & 4.6e-11 (7.37) \\
  10 & 5.9e-06 & 1.8e-07 (9.69) & 1.9e-08 (9.01) & 1.4e-09 (9.11) & 8.1e-11 (7.52) \\
  10 & 1.3e-05 & 4.4e-07 (9.42) & 2.4e-08 (11.57) & 2.0e-09 (8.73) & 1.1e-10 (7.71) \\
  13 & 2.5e-05 & 1.7e-06 (7.40) & 7.6e-08 (12.27) & 5.7e-09 (9.22) & 4.2e-10 (6.83) \\
  13 & 7.5e-05 & 2.6e-06 (9.34) & 1.2e-07 (12.17) & 1.1e-08 (8.62) & 5.0e-10 (7.99) \\
  17 & 8.1e-05 & 2.7e-06 (9.42) & 2.6e-07 (9.21) & 2.1e-08 (9.03) & 9.1e-10 (8.21) \\
  17 & 1.1e-04 & 4.9e-06 (8.74) & 3.2e-07 (10.79) & 2.3e-08 (9.38) & 1.2e-09 (7.65) \\
\hline
$P$ &\multicolumn{1}{c}{50} &\multicolumn{1}{c}{100}&\multicolumn{1}{c}{200} &\multicolumn{1}{c}{400} &\multicolumn{1}{c}{800} \\
\hline
\end{tabular}
\end{table}

\subsection{Hexagonal meshes}
\label{sec:hex}

\begin{table}
\begin{center}
\caption{Dimension of the kernel of matrix $\m{B}$ for hexagonal mesh $\mathcal{H}$}
\label{tb:kernelH}
\begin{tabular}{r|r r| r r| r r| r r}
$\mathcal{H}$ &\multicolumn{6}{c}{ $\dim(\Kb)$ ($\dim(V_h)$)}\\[2pt]
\hline
\texttt{n x m}& \multicolumn{2}{c}{$k=1$} & \multicolumn{2}{c}{$k=2$} &
\multicolumn{2}{c}{$k=3$}& \multicolumn{2}{c}{$k=4$}\\[2pt]
\hline
 \texttt{ 8 x 10} & 0  & (150)     & 0  & (487)     & 8    & (918)          &
53    & (1443)\\
 \texttt{18 x 20} & 0  & (700)     & 0  & (2177)   & 187    & (4043)    & 498
& (6298)\\
 \texttt{26 x 30} & 0  & (1530)   & 0  & (4703)   & 549  & (8698)      & 1257 & (13515)\\
 \texttt{34 x 40} & 0  & (2680)   & 0  & (8189)   & 1016  & (15113)  & 2328  & (23452)\\
 \texttt{44 x 50} & 0  & (4350)   & 0  & (13239) & 1796  & (24398)  & 3974 & (37827)\\
 \texttt{52 x 60} & 0  & (6180)   & 0  & (18765) & 2693  & (34553)  &  5764
& (53544)\\
 \texttt{60 x 70} & 0  & (8330)   & 1  & (25251) & 3670  & (46468)         & 7912
& (71981)\\
 \texttt{70 x 80} & 0  & (11120)  &1  & (33661)  & 5036  & (61913)         & 10670
& (95876)\\[2pt]
\hline
\end{tabular}
\end{center}
\end{table}

Table~\ref{tb:kernelH} displays the dimensions of $\Kb$ and $V_h$ for regular
hexagonal meshes. We observe that for $k=1$ we have again that the kernel is
always reduced to $\{0\}$, for $k=2$ the kernel contains one element for the
two finest meshes. For $k=3,4$ we have a situation similar to the case of
Voronoi meshes, that is the dimension of $\Kb$ is positive and the first ten
computed eigenvalues converge optimally, see Tables~\ref{tb:hk1}-\ref{tb:hk4}.
We can see that for $k=4$ the rate of convergence is highly oscillating. This
is due to the fact that in this case the error is close to machine precision.
\begin{table}
\caption{ First 10 eigenvalues on $\mathcal{H}$ with $k=1$ }
\label{tb:hk1}
\begin{tabular}{r|rrrrr}
\hline
Exact&\multicolumn{5}{c}{Errors (rate)}\\[3pt]
\hline
   2 & 2.2e-02 & 4.9e-03 (2.15) & 1.3e-03 (1.91) & 6.0e-04 (1.91) & 3.0e-04 (2.41) \\
   5 & 1.2e-01 & 2.9e-02 (2.11) & 7.8e-03 (1.88) & 3.5e-03 (1.98) & 1.9e-03 (2.12) \\
   5 & 1.5e-01 & 3.2e-02 (2.28) & 8.2e-03 (1.95) & 3.5e-03 (2.10) & 2.0e-03 (1.95) \\
   8 & 3.5e-01 & 7.8e-02 (2.18) & 2.0e-02 (1.93) & 8.9e-03 (2.05) & 5.0e-03 (2.00) \\
  10 & 4.8e-01 & 1.1e-01 (2.10) & 3.1e-02 (1.87) & 1.4e-02 (1.97) & 7.6e-03 (2.10) \\
  10 & 6.5e-01 & 1.3e-01 (2.33) & 3.3e-02 (1.97) & 1.4e-02 (2.13) & 7.9e-03 (1.96) \\
  13 & 8.8e-01 & 2.0e-01 (2.14) & 5.3e-02 (1.91) & 2.3e-02 (2.02) & 1.3e-02 (2.04) \\
  13 & 1.0e+00 & 2.1e-01 (2.25) & 5.5e-02 (1.96) & 2.4e-02 (2.08) & 1.3e-02 (2.00) \\
  17 & 1.4e+00 & 3.2e-01 (2.11) & 8.9e-02 (1.87) & 4.0e-02 (1.97) & 2.2e-02 (2.10) \\
  17 & 2.0e+00 & 3.8e-01 (2.38) & 9.6e-02 (1.97) & 4.0e-02 (2.14) & 2.3e-02 (1.97) \\
\hline
\texttt{m} &\multicolumn{1}{c}{10} &\multicolumn{1}{c}{20}&\multicolumn{1}{c}{40} &\multicolumn{1}{c}{60} &\multicolumn{1}{c}{80} \\
\hline
\end{tabular}
\end{table}
\begin{table}
\caption{ First 10 eigenvalues on $\mathcal{H}$ with $k=2$ }
\label{tb:hk2}
\begin{tabular}{r|rrrrr}
\hline
Exact&\multicolumn{5}{c}{Errors (rate)}\\[3pt]
\hline
   2 & 1.1e-04 & 4.9e-06 (4.46) & 3.9e-07 (3.66) & 7.3e-08 (4.14) & 2.2e-08 (4.09) \\
   5 & 1.1e-03 & 6.4e-05 (4.12) & 4.7e-06 (3.78) & 9.1e-07 (4.04) & 2.9e-07 (4.02) \\
   5 & 1.9e-03 & 9.4e-05 (4.34) & 7.2e-06 (3.71) & 1.4e-06 (4.11) & 4.2e-07 (4.08) \\
   8 & 6.7e-03 & 3.1e-04 (4.42) & 2.5e-05 (3.65) & 4.6e-06 (4.13) & 1.4e-06 (4.09) \\
  10 & 8.9e-03 & 6.1e-04 (3.87) & 4.2e-05 (3.86) & 8.3e-06 (3.98) & 2.6e-06 (3.99) \\
  10 & 1.3e-02 & 6.8e-04 (4.22) & 5.1e-05 (3.74) & 9.8e-06 (4.09) & 3.0e-06 (4.06) \\
  13 & 2.0e-02 & 1.1e-03 (4.24) & 8.4e-05 (3.68) & 1.6e-05 (4.09) & 5.0e-06 (4.06) \\
  13 & 3.3e-02 & 1.6e-03 (4.35) & 1.3e-04 (3.67) & 2.4e-05 (4.12) & 7.3e-06 (4.09) \\
  17 & 4.4e-02 & 3.1e-03 (3.86) & 2.2e-04 (3.77) & 4.4e-05 (4.03) & 1.4e-05 (4.05) \\
  17 & 5.2e-02 & 3.3e-03 (3.98) & 2.3e-04 (3.86) & 4.5e-05 (4.00) & 1.4e-05 (3.97) \\
\hline
\texttt{m} &\multicolumn{1}{c}{10} &\multicolumn{1}{c}{20}&\multicolumn{1}{c}{40} &\multicolumn{1}{c}{60} &\multicolumn{1}{c}{80} \\
\hline
\end{tabular}
\end{table}
\begin{table}
\caption{ First 10 eigenvalues on $\mathcal{H}$ with $k=3$ }
\label{tb:hk3}
\begin{tabular}{r|rrrrr}
\hline
Exact&\multicolumn{5}{c}{Errors (rate)}\\[3pt]
\hline
   2 & 2.4e-07 & 3.3e-09 (6.21) & 6.0e-11 (5.78) & 5.2e-12 (6.04) & 1.0e-12 (5.73) \\
   5 & 1.0e-05 & 9.8e-08 (6.74) & 2.1e-09 (5.57) & 1.7e-10 (6.21) & 2.8e-11 (6.13) \\
   5 & 1.2e-05 & 1.5e-07 (6.22) & 2.8e-09 (5.79) & 2.4e-10 (6.06) & 4.2e-11 (6.04) \\
   8 & 6.0e-05 & 8.3e-07 (6.16) & 1.5e-08 (5.77) & 1.3e-09 (6.05) & 2.3e-10 (6.03) \\
  10 & 1.6e-04 & 1.8e-06 (6.55) & 3.8e-08 (5.52) & 3.0e-09 (6.25) & 5.2e-10 (6.18) \\
  10 & 2.0e-04 & 2.3e-06 (6.47) & 4.1e-08 (5.81) & 3.5e-09 (6.06) & 6.1e-10 (6.04) \\
  13 & 3.8e-04 & 4.5e-06 (6.41) & 9.0e-08 (5.65) & 7.5e-09 (6.13) & 1.3e-09 (6.09) \\
  13 & 4.9e-04 & 6.9e-06 (6.15) & 1.3e-07 (5.78) & 1.1e-08 (6.05) & 1.9e-09 (6.04) \\
  17 & 1.2e-03 & 1.6e-05 (6.19) & 3.1e-07 (5.70) & 2.7e-08 (6.05) & 4.7e-09 (6.04) \\
  17 & 1.8e-03 & 1.8e-05 (6.68) & 3.5e-07 (5.64) & 2.8e-08 (6.25) & 4.7e-09 (6.18) \\
\hline
\texttt{m} &\multicolumn{1}{c}{10} &\multicolumn{1}{c}{20}&\multicolumn{1}{c}{40} &\multicolumn{1}{c}{60} &\multicolumn{1}{c}{80} \\
\hline
\end{tabular}
\end{table}
\begin{table}
\caption{ First 10 eigenvalues on $\mathcal{H}$ with $k=4$ }
\label{tb:hk4}
\begin{tabular}{r|rrrrr}
\hline
Exact&\multicolumn{5}{c}{Errors (rate)}\\[3pt]
\hline
   2 & 3.8e-10 & 6.0e-13 (9.30) & 2.7e-13 (1.15) & 8.1e-14 (2.99) & 6.4e-13 (-7.20) \\
   5 & 7.3e-08 & 1.8e-10 (8.63) & 1.2e-12 (7.23) & 1.3e-12 (-0.14) & 1.0e-13 (8.87) \\
   5 & 1.8e-08 & 4.8e-11 (8.53) & 4.2e-13 (6.82) & 4.6e-13 (-0.21) & 6.8e-13 (-1.37) \\
   8 & 3.5e-07 & 6.5e-10 (9.05) & 4.5e-12 (7.19) & 8.5e-13 (4.09) & 7.2e-13 (0.58) \\
  10 & 1.1e-06 & 3.7e-09 (8.19) & 2.1e-11 (7.48) & 1.3e-12 (6.88) & 1.1e-13 (8.42) \\
  10 & 6.1e-09 & 1.1e-09 (2.46) & 4.7e-12 (7.88) & 7.1e-15 (16.02) & 1.3e-12 (-18.20) \\
  13 & 8.3e-06 & 2.1e-08 (8.64) & 1.3e-10 (7.36) & 4.5e-12 (8.22) & 3.7e-14 (16.68) \\
  13 & 1.7e-06 & 6.1e-09 (8.15) & 2.1e-11 (8.16) & 3.2e-13 (10.38) & 9.2e-13 (-3.70) \\
  17 & 6.2e-06 & 4.3e-08 (7.15) & 1.9e-10 (7.81) & 8.4e-12 (7.72) & 1.5e-13 (14.02) \\
  17 & 3.1e-06 & 2.8e-08 (6.82) & 1.5e-10 (7.53) & 5.9e-12 (7.99) & 5.1e-13 (8.49) \\
\hline
\texttt{m} &\multicolumn{1}{c}{10} &\multicolumn{1}{c}{20}&\multicolumn{1}{c}{40} &\multicolumn{1}{c}{60} &\multicolumn{1}{c}{80} \\
\hline
\end{tabular}
\end{table}

\subsection{Dyadic meshes}
\label{sec:dyadic}

We conclude by considering a less standard mesh, called \emph{dyadic} and labeled by $\mathcal{D}$, obtained by dividing the unit square into $N^2$ squares. Each square is considered as a 
polygon with 8 edges by adding the midpont to each edge. Therefore the number of vertices is $(3N+1)(N+1)$. The mesh corresponding to 
$N=4$ is depicted in Figure~\ref{fig:dyadic}, where we highlighted the vertices and the edges.
\begin{figure}
\begin{center}
\includegraphics[width=0.22\textwidth]{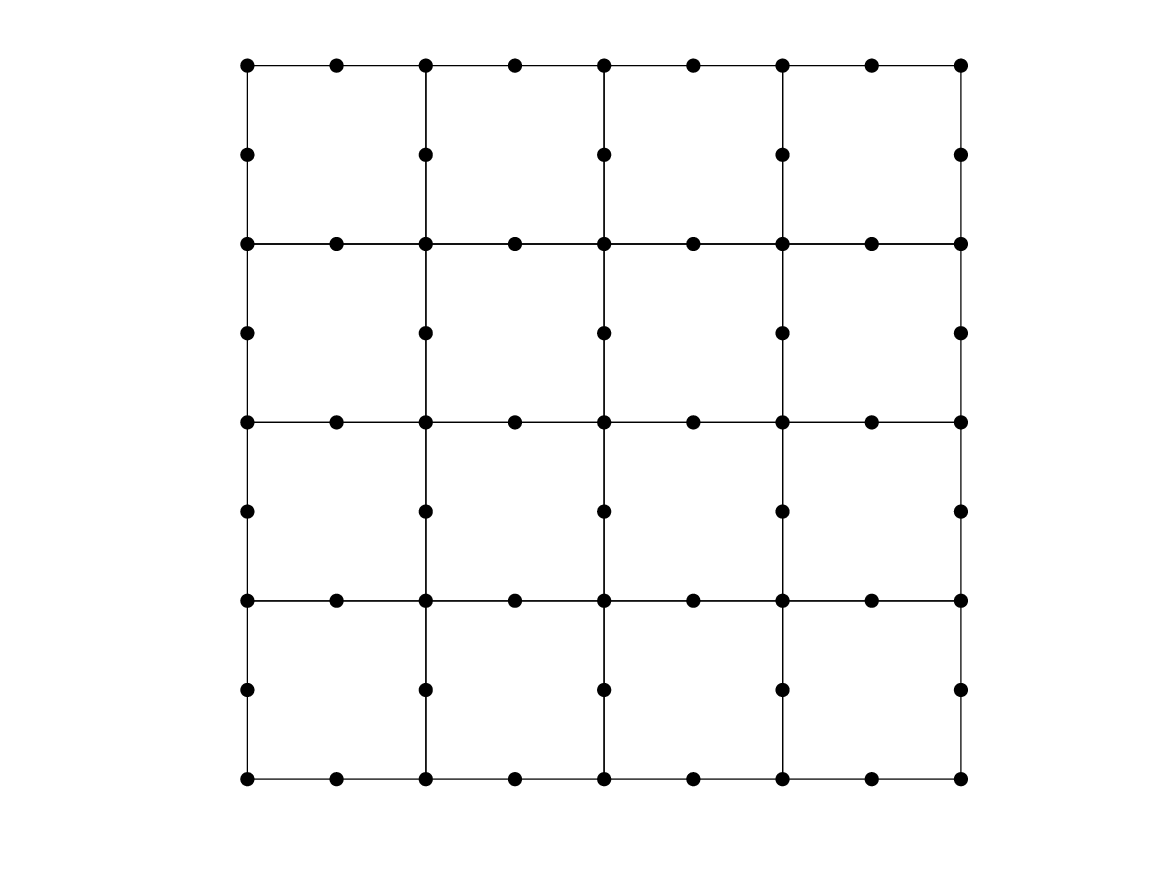}   
\end{center}
\caption{Coarsest dyadic mesh $\mathcal{D}$}
\label{fig:dyadic}
\end{figure} 
Table~\ref{tb:kernelD} displays the dimensions of $\Kb$ and of $V_h$
corresponding to different values of $N$ ranging from $4$ to $64$, 
and of the degree $k$ from $1$ to $4$. In this case we have that the kernel is
not trivial for all the degrees of the polynomials contained in the VEM space,
however we have no spurious eigenvalues. This is predicted by our theory for
the cases $k=1,2$, while for $k=3,4$ we have only numerical evidence. Moreover,
the rate of convergence is always optimal, as it can be seen in
Tables~\ref{tb:dk1}-\ref{tb:dk4}. As in the case of hexagonal meshes, the rate
of convergence for $k=4$ is oscillating due to the fact that the errors are
close to machine precision.
\begin{table}
\begin{center}
\caption{Dimension of the kernel of matrix $\m{B}$ for dyadic mesh
$\mathcal{D}$}
\label{tb:kernelD}
\begin{tabular}{r|r r| r r| r r| r r}
$\mathcal{D}$ &\multicolumn{6}{c}{ $\dim(\Kb)$ ($\dim(V_h)$)}\\[2pt]
\hline
N& \multicolumn{2}{c}{$k=1$} & \multicolumn{2}{c}{$k=2$} &  
\multicolumn{2}{c}{$k=3$}& \multicolumn{2}{c}{$k=4$}\\[2pt]
\hline
 4 &  9 & (33)        & 42 & (97)    &  66 & (177) & 90 & (273) \\
 8 & 49 & (161)      & 210 & (449)  & 322 & (801) & 434 & (1217)\\
16 & 225 & (705)    & 930 & (1921)  & 1410 & (3393) & 1890 & (5121)\\
32 & 961 & (2945)   & 3906 & (7937)  & 5890 & (13953) & 7874 & (20993)\\
64 & 3969 & (12033)  & 16002 & (32257) & 24066 & (56577) & 32131 & (84993)\\
\hline
\end{tabular}
\end{center}
\end{table}
\begin{table}
\caption{ First 10 eigenvalues on $\mathcal{D}$ with $k=1$ }
\label{tb:dk1}
\begin{tabular}{r|rrrrr}
\hline
Exact&\multicolumn{5}{c}{Errors (rate)}\\[3pt]
\hline
   2 & 1.7e-01 & 3.9e-02 (2.08) & 9.7e-03 (2.02) & 2.4e-03 (2.01) & 6.0e-04 (2.00) \\
   5 & 9.3e-01 & 2.0e-01 (2.23) & 4.7e-02 (2.08) & 1.1e-02 (2.02) & 2.9e-03 (1.99) \\
   5 & 9.3e-01 & 2.0e-01 (2.23) & 4.7e-02 (2.08) & 1.1e-02 (2.02) & 2.9e-03 (1.99) \\
   8 & 3.0e+00 & 6.6e-01 (2.19) & 1.6e-01 (2.08) & 3.9e-02 (2.02) & 9.7e-03 (2.00) \\
  10 & 3.5e+00 & 6.8e-01 (2.36) & 1.5e-01 (2.21) & 3.5e-02 (2.07) & 8.7e-03 (2.01) \\
  10 & 3.5e+00 & 6.8e-01 (2.36) & 1.5e-01 (2.21) & 3.5e-02 (2.07) & 8.7e-03 (2.01) \\
  13 & 7.6e+00 & 1.7e+00 (2.17) & 3.8e-01 (2.14) & 9.3e-02 (2.04) & 2.3e-02 (2.01) \\
  13 & 7.6e+00 & 1.7e+00 (2.17) & 3.8e-01 (2.14) & 9.3e-02 (2.04) & 2.3e-02 (2.01) \\
  17 & 1.5e+01 & 1.9e+00 (2.94) & 3.8e-01 (2.35) & 8.7e-02 (2.13) & 2.1e-02 (2.04) \\
  17 & 2.4e+01 & 1.9e+00 (3.61) & 3.8e-01 (2.35) & 8.7e-02 (2.13) & 2.1e-02 (2.04) \\
\hline
$N$ &\multicolumn{1}{c}{4} &\multicolumn{1}{c}{8}&\multicolumn{1}{c}{16} &\multicolumn{1}{c}{32} &\multicolumn{1}{c}{64} \\
\hline
\end{tabular}
\end{table}
\begin{table}
\caption{ First 10 eigenvalues on $\mathcal{D}$ with $k=2$ }
\label{tb:dk2}
\begin{tabular}{r|rrrrr}
\hline
Exact&\multicolumn{5}{c}{Errors (rate)}\\[3pt]
\hline
   2 & 8.6e-04 & 4.8e-05 (4.16) & 2.9e-06 (4.04) & 1.8e-07 (4.01) & 1.1e-08 (4.00) \\
   5 & 3.3e-02 & 2.0e-03 (4.05) & 1.2e-04 (4.01) & 7.6e-06 (4.00) & 4.8e-07 (4.00) \\
   5 & 3.3e-02 & 2.0e-03 (4.05) & 1.2e-04 (4.01) & 7.6e-06 (4.00) & 4.8e-07 (4.00) \\
   8 & 7.8e-02 & 3.4e-03 (4.50) & 1.9e-04 (4.16) & 1.2e-05 (4.04) & 7.3e-07 (4.01) \\
  10 & 3.3e-01 & 2.3e-02 (3.89) & 1.4e-03 (3.97) & 9.1e-05 (3.99) & 5.7e-06 (4.00) \\
  10 & 3.3e-01 & 2.3e-02 (3.89) & 1.4e-03 (3.97) & 9.1e-05 (3.99) & 5.7e-06 (4.00) \\
  13 & 4.8e-01 & 2.4e-02 (4.32) & 1.4e-03 (4.12) & 8.5e-05 (4.03) & 5.3e-06 (4.01) \\
  13 & 4.8e-01 & 2.4e-02 (4.32) & 1.4e-03 (4.12) & 8.5e-05 (4.03) & 5.3e-06 (4.01) \\
  17 & 5.0e-01 & 1.2e-01 (2.03) & 8.1e-03 (3.92) & 5.2e-04 (3.98) & 3.2e-05 (3.99) \\
  17 & 5.0e-01 & 1.2e-01 (2.03) & 8.1e-03 (3.92) & 5.2e-04 (3.98) & 3.2e-05 (3.99) \\
\hline
$N$ &\multicolumn{1}{c}{4} &\multicolumn{1}{c}{8}&\multicolumn{1}{c}{16} &\multicolumn{1}{c}{32} &\multicolumn{1}{c}{64} \\
\hline
\end{tabular}
\end{table}
\begin{table}
\caption{ First 10 eigenvalues on $\mathcal{D}$ with $k=3$ }
\label{tb:dk3}
\begin{tabular}{r|rrrrr}
\hline
Exact&\multicolumn{5}{c}{Errors (rate)}\\[3pt]
\hline
   2 & 8.1e-05 & 1.3e-06 (5.97) & 2.0e-08 (5.99) & 3.2e-10 (6.00) & 3.9e-12 (6.35) \\
   5 & 1.8e-03 & 3.3e-05 (5.80) & 5.3e-07 (5.96) & 8.4e-09 (5.99) & 1.3e-10 (6.02) \\
   5 & 1.8e-03 & 3.3e-05 (5.80) & 5.3e-07 (5.96) & 8.4e-09 (5.99) & 1.3e-10 (6.00) \\
   8 & 1.8e-02 & 3.2e-04 (5.79) & 5.2e-06 (5.97) & 8.1e-08 (5.99) & 1.3e-09 (6.00) \\
  10 & 1.8e-02 & 3.8e-04 (5.53) & 6.5e-06 (5.89) & 1.0e-07 (5.97) & 1.6e-09 (5.99) \\
  10 & 1.8e-02 & 3.8e-04 (5.53) & 6.5e-06 (5.89) & 1.0e-07 (5.97) & 1.6e-09 (5.99) \\
  13 & 8.3e-02 & 1.8e-03 (5.51) & 3.0e-05 (5.91) & 4.8e-07 (5.98) & 7.5e-09 (6.00) \\
  13 & 8.3e-02 & 1.8e-03 (5.51) & 3.0e-05 (5.91) & 4.8e-07 (5.98) & 7.5e-09 (6.00) \\
  17 & 2.2e-01 & 2.8e-03 (6.31) & 4.9e-05 (5.82) & 7.8e-07 (5.96) & 1.2e-08 (5.99) \\
  17 & 2.2e-01 & 2.8e-03 (6.31) & 4.9e-05 (5.82) & 7.8e-07 (5.96) & 1.2e-08 (5.99) \\
\hline
$N$ &\multicolumn{1}{c}{4} &\multicolumn{1}{c}{8}&\multicolumn{1}{c}{16} &\multicolumn{1}{c}{32} &\multicolumn{1}{c}{64} \\
\hline
\end{tabular}
\end{table}
\begin{table}
\caption{ First 10 eigenvalues on $\mathcal{D}$ with $k=4$ }
\label{tb:dk4}
\begin{tabular}{r|rrrrr}
\hline
Exact&\multicolumn{5}{c}{Errors (rate)}\\[3pt]
\hline
   2 & 6.9e-07 & 2.9e-09 (7.87) & 1.1e-11 (8.02) & 2.8e-12 (2.05) & 2.5e-12 (0.11) \\
   5 & 4.8e-05 & 2.2e-07 (7.78) & 8.9e-10 (7.94) & 5.6e-13 (10.64) & 1.1e-12 (-1.02) \\
   5 & 4.8e-05 & 2.2e-07 (7.78) & 8.9e-10 (7.94) & 1.9e-12 (8.84) & 5.4e-12 (-1.47) \\
   8 & 5.1e-04 & 2.8e-06 (7.54) & 1.2e-08 (7.87) & 4.4e-11 (8.06) & 6.1e-12 (2.87) \\
  10 & 1.0e-03 & 4.9e-06 (7.70) & 2.0e-08 (7.92) & 7.8e-11 (8.04) & 7.3e-13 (6.73) \\
  10 & 1.0e-03 & 4.9e-06 (7.70) & 2.0e-08 (7.92) & 7.9e-11 (8.02) & 7.3e-12 (3.43) \\
  13 & 4.6e-03 & 2.7e-05 (7.40) & 1.2e-07 (7.82) & 4.8e-10 (7.96) & 4.3e-12 (6.82) \\
  13 & 4.6e-03 & 2.7e-05 (7.40) & 1.2e-07 (7.82) & 4.8e-10 (7.96) & 1.4e-11 (5.12) \\
  17 & 1.2e-02 & 5.4e-05 (7.82) & 2.3e-07 (7.89) & 9.0e-10 (7.98) & 3.1e-12 (8.18) \\
  17 & 1.2e-02 & 5.4e-05 (7.82) & 2.3e-07 (7.89) & 9.0e-10 (7.98) & 8.0e-12 (6.81) \\
\hline
$N$ &\multicolumn{1}{c}{4} &\multicolumn{1}{c}{8}&\multicolumn{1}{c}{16} &\multicolumn{1}{c}{32} &\multicolumn{1}{c}{64} \\
\hline
\end{tabular}
\end{table}

\section*{Acknowledgments}

The authors are member of INdAM Research group GNCS.

\bibliography{ref.bib}

@article {DNR1,
    AUTHOR = {Descloux, Jean and Nassif, Nabil and Rappaz, Jacques},
     TITLE = {On spectral approximation. {I}. {T}he problem of convergence},
   JOURNAL = {RAIRO Anal. Num\'er.},
  FJOURNAL = {RAIRO Analyse Num\'erique},
    VOLUME = {12},
      YEAR = {1978},
    NUMBER = {2},
     PAGES = {97--112, iii},
      ISSN = {0399-0516,0516-2777},
   MRCLASS = {65J05},
  MRNUMBER = {483400},
MRREVIEWER = {Juhani\ Pitk\"aranta},
       DOI = {10.1051/m2an/1978120200971},
       URL = {https://doi.org/10.1051/m2an/1978120200971},
}

@article {DNR2,
    AUTHOR = {Descloux, Jean and Nassif, Nabil and Rappaz, Jacques},
     TITLE = {On spectral approximation. {II}. {E}rror estimates for the
              {G}alerkin method},
   JOURNAL = {RAIRO Anal. Num\'er.},
  FJOURNAL = {RAIRO Analyse Num\'erique},
    VOLUME = {12},
      YEAR = {1978},
    NUMBER = {2},
     PAGES = {113--119, iii},
      ISSN = {0399-0516,0516-2777},
   MRCLASS = {65J05},
  MRNUMBER = {483401},
MRREVIEWER = {Juhani\ Pitk\"aranta},
       DOI = {10.1051/m2an/1978120201131},
       URL = {https://doi.org/10.1051/m2an/1978120201131},
}

@article {auto-stab,
    AUTHOR = {Boffi, Daniele and Gardini, Francesca and Gastaldi, Lucia},
     TITLE = {Approximation of {PDE} eigenvalue problems involving parameter
              dependent matrices},
   JOURNAL = {Calcolo},
  FJOURNAL = {Calcolo. A Quarterly on Numerical Analysis and Theory of
              Computation},
    VOLUME = {57},
      YEAR = {2020},
    NUMBER = {4},
     PAGES = {Paper No. 41, 21},
      ISSN = {0008-0624,1126-5434},
   MRCLASS = {65N30 (65N25)},
  MRNUMBER = {4177014},
MRREVIEWER = {Mohammad\ Asadzadeh},
       DOI = {10.1007/s10092-020-00390-6},
       URL = {https://doi.org/10.1007/s10092-020-00390-6},
}

@article {Fra-VEM,
    AUTHOR = {Gardini, Francesca and Vacca, Giuseppe},
     TITLE = {Virtual element method for second-order elliptic eigenvalue
              problems},
   JOURNAL = {IMA J. Numer. Anal.},
  FJOURNAL = {IMA Journal of Numerical Analysis},
    VOLUME = {38},
      YEAR = {2018},
    NUMBER = {4},
     PAGES = {2026--2054},
      ISSN = {0272-4979,1464-3642},
   MRCLASS = {65N25 (65N30)},
  MRNUMBER = {3867390},
       DOI = {10.1093/imanum/drx063},
       URL = {https://doi.org/10.1093/imanum/drx063},
}

@article {enhanced,
    AUTHOR = {Ahmad, B. and Alsaedi, A. and Brezzi, F. and Marini, L. D. and
              Russo, A.},
     TITLE = {Equivalent projectors for virtual element methods},
   JOURNAL = {Comput. Math. Appl.},
  FJOURNAL = {Computers \& Mathematics with Applications. An International
              Journal},
    VOLUME = {66},
      YEAR = {2013},
    NUMBER = {3},
     PAGES = {376--391},
      ISSN = {0898-1221},
   MRCLASS = {65N30 (65N12)},
  MRNUMBER = {3073346},
MRREVIEWER = {Francesco Calabr\`o},
       DOI = {10.1016/j.camwa.2013.05.015},
}

@article {basic,
    AUTHOR = {Beir\~ao da Veiga, L. and Brezzi, F. and Cangiani, A. and
              Manzini, G. and Marini, L. D. and Russo, A.},
     TITLE = {Basic principles of virtual element methods},
   JOURNAL = {Math. Models Methods Appl. Sci.},
  FJOURNAL = {Mathematical Models and Methods in Applied Sciences},
    VOLUME = {23},
      YEAR = {2013},
    NUMBER = {1},
     PAGES = {199--214},
      ISSN = {0218-2025,1793-6314},
   MRCLASS = {65N06},
  MRNUMBER = {2997471},
MRREVIEWER = {Bo\v sko\ S.\ Jovanovi\'c},
       DOI = {10.1142/S0218202512500492},
       URL = {https://doi.org/10.1142/S0218202512500492},
}

@article {Acta-VEM,
    AUTHOR = {Beir\~ao da Veiga, Louren\c{c}o and Brezzi, Franco and Marini,
              L. Donatella and Russo, Alessandro},
     TITLE = {The virtual element method},
   JOURNAL = {Acta Numer.},
  FJOURNAL = {Acta Numerica},
    VOLUME = {32},
      YEAR = {2023},
     PAGES = {123--202},
      ISSN = {0962-4929,1474-0508},
   MRCLASS = {65N30},
  MRNUMBER = {4586821},
MRREVIEWER = {Feng\ Wang},
       DOI = {10.1017/S0962492922000095},
       URL = {https://doi.org/10.1017/S0962492922000095},
}

@article{berrone2023stabilizationfree,
    AUTHOR = {Berrone, Stefano and Borio, Andrea and Marcon, Francesca},
     TITLE = {Lowest order stabilization free virtual element method for the
              2{D} {P}oisson equation},
   JOURNAL = {Comput. Math. Appl.},
  FJOURNAL = {Computers \& Mathematics with Applications. An International
              Journal},
    VOLUME = {177},
      YEAR = {2025},
     PAGES = {78--99},
      ISSN = {0898-1221,1873-7668},
   MRCLASS = {65N12 (65N15 65N30)},
  MRNUMBER = {4828369},
MRREVIEWER = {Nasser\ H.\ Sweilam},
       DOI = {10.1016/j.camwa.2024.11.017},
       URL = {https://doi.org/10.1016/j.camwa.2024.11.017},
}

@article {Mengetal,
    AUTHOR = {Meng, Jian and Wang, Xue and Bu, Linlin and Mei, Liquan},
     TITLE = {A lowest-order free-stabilization virtual element method for
              the {L}aplacian eigenvalue problem},
   JOURNAL = {J. Comput. Appl. Math.},
  FJOURNAL = {Journal of Computational and Applied Mathematics},
    VOLUME = {410},
      YEAR = {2022},
     PAGES = {Paper No. 114013, 11},
      ISSN = {0377-0427,1879-1778},
   MRCLASS = {65N25 (65N15 65N30)},
  MRNUMBER = {4386637},
MRREVIEWER = {Anh-Khoa\ Vo},
       DOI = {10.1016/j.cam.2021.114013},
       URL = {https://doi.org/10.1016/j.cam.2021.114013},
}

@misc{CAVE,
  author = {{Team at University of Milano-Bicocca}},
  title = {Virtual Element @ Bicocca},
  howpublished = { \url{https://https://sites.google.com/view/vembic/home}},
  year = {2016},
}

@article {Boffi-acta,
AUTHOR = {Boffi, Daniele},
TITLE = {Finite element approximation of eigenvalue problems},
JOURNAL = {Acta Numer.},
FJOURNAL = {Acta Numerica},
VOLUME = {19},
YEAR = {2010},
PAGES = {1--120},
ISSN = {0962-4929,1474-0508},
MRCLASS = {65N30 (65N25)},
MRNUMBER = {2652780},
MRREVIEWER = {Srinivasan\ Kesavan},
DOI = {10.1017/S0962492910000012},
URL = {https://doi.org/10.1017/S0962492910000012},
}

@incollection {vem-springer,
    AUTHOR = {Boffi, Daniele and Gardini, Francesca and Gastaldi, Lucia},
     TITLE = {Virtual element approximation of eigenvalue problems},
 BOOKTITLE = {The virtual element method and its applications},
    SERIES = {SEMA SIMAI Springer Ser.},
    VOLUME = {31},
     PAGES = {275--320},
 PUBLISHER = {Springer, Cham},
      YEAR = {[2022] \copyright 2022},
      ISBN = {978-3-030-95318-8; 978-3-030-95319-5},
   MRCLASS = {65N21 (65N30)},
  MRNUMBER = {4510905},
       DOI = {10.1007/978-3-030-95319-5\_7},
       URL = {https://doi.org/10.1007/978-3-030-95319-5_7},
}

@article {fra_nonconf,
    AUTHOR = {Gardini, Francesca and Manzini, Gianmarco and Vacca, Giuseppe},
     TITLE = {The nonconforming virtual element method for eigenvalue
              problems},
   JOURNAL = {ESAIM Math. Model. Numer. Anal.},
  FJOURNAL = {ESAIM. Mathematical Modelling and Numerical Analysis},
    VOLUME = {53},
      YEAR = {2019},
    NUMBER = {3},
     PAGES = {749--774},
      ISSN = {2822-7840,2804-7214},
   MRCLASS = {65N30 (65N25)},
  MRNUMBER = {3959470},
MRREVIEWER = {Marco\ Picasso},
       DOI = {10.1051/m2an/2018074},
       URL = {https://doi.org/10.1051/m2an/2018074},
}

@article {mora_steklov,
    AUTHOR = {Mora, David and Rivera, Gonzalo and Rodr\'iguez, Rodolfo},
     TITLE = {A virtual element method for the {S}teklov eigenvalue problem},
   JOURNAL = {Math. Models Methods Appl. Sci.},
  FJOURNAL = {Mathematical Models and Methods in Applied Sciences},
    VOLUME = {25},
      YEAR = {2015},
    NUMBER = {8},
     PAGES = {1421--1445},
      ISSN = {0218-2025,1793-6314},
   MRCLASS = {65N25 (35J25 65N30)},
  MRNUMBER = {3340705},
MRREVIEWER = {Michael\ Karkulik},
       DOI = {10.1142/S0218202515500372},
       URL = {https://doi.org/10.1142/S0218202515500372},
}

@article {Berronefirst,
    AUTHOR = {Berrone, Stefano and Borio, Andrea and Marcon, Francesca and
              Teora, Gioana},
     TITLE = {A first-order stabilization-free virtual element method},
   JOURNAL = {Appl. Math. Lett.},
  FJOURNAL = {Applied Mathematics Letters. An International Journal of Rapid
              Publication},
    VOLUME = {142},
      YEAR = {2023},
     PAGES = {Paper No. 108641, 6},
      ISSN = {0893-9659,1873-5452},
   MRCLASS = {65M60},
  MRNUMBER = {4563532},
       DOI = {10.1016/j.aml.2023.108641},
       URL = {https://doi.org/10.1016/j.aml.2023.108641},
}

@article {self-stabilized,
    AUTHOR = {Lamperti, Andrea and Cremonesi, Massimiliano and Perego,
              Umberto and Russo, Alessandro and Lovadina, Carlo},
     TITLE = {A {H}u-{W}ashizu variational approach to self-stabilized
              virtual elements: 2{D} linear elastostatics},
   JOURNAL = {Comput. Mech.},
  FJOURNAL = {Computational Mechanics},
    VOLUME = {71},
      YEAR = {2023},
    NUMBER = {5},
     PAGES = {935--955},
      ISSN = {0178-7675,1432-0924},
   MRCLASS = {65M60 (74S05)},
  MRNUMBER = {4577379},
       DOI = {10.1007/s00466-023-02282-2},
       URL = {https://doi.org/10.1007/s00466-023-02282-2},
}

@article {Paola,
    AUTHOR = {Foligno, Paola Pia and Boffi, Daniele and Credali, Fabio and
              Vescovini, Riccardo},
     TITLE = {Benchmarking stabilized and self-stabilized $p$-virtual element
              methods with variable coefficients},
   JOURNAL = {Comput. Methods Appl. Mech. Engrg.},
    VOLUME = {},
      YEAR = {},
    NUMBER = {},
     PAGES = {},
      NOTE = {To appear},
}

@article {Linda,
    AUTHOR = {Alzaben, Linda and Boffi, Daniele and Dedner, Andreas and
              Gastaldi, Lucia},
     TITLE = {On the stabilization of a virtual element method for an
              acoustic vibration problem},
   JOURNAL = {Math. Models Methods Appl. Sci.},
  FJOURNAL = {Mathematical Models and Methods in Applied Sciences},
    VOLUME = {35},
      YEAR = {2025},
    NUMBER = {3},
     PAGES = {655--701},
      ISSN = {0218-2025,1793-6314},
   MRCLASS = {65N30 (65N25 70J30 76M10)},
  MRNUMBER = {4884682},
       DOI = {10.1142/S0218202525500071},
       URL = {https://doi.org/10.1142/S0218202525500071},
}

@article {ChenSukumar,
    AUTHOR = {Chen, Alvin and Sukumar, N.},
     TITLE = {Stabilization-free virtual element method for plane
              elasticity},
   JOURNAL = {Comput. Math. Appl.},
  FJOURNAL = {Computers \& Mathematics with Applications. An International
              Journal},
    VOLUME = {138},
      YEAR = {2023},
     PAGES = {88--105},
      ISSN = {0898-1221,1873-7668},
   MRCLASS = {74S05 (65N30 74Bxx)},
  MRNUMBER = {4561596},
MRREVIEWER = {Carlos\ Henrique\ Daros},
       DOI = {10.1016/j.camwa.2023.03.002},
       URL = {https://doi.org/10.1016/j.camwa.2023.03.002},
}

@article {MarconMora,
    AUTHOR = {Marcon, Francesca and Mora, David},
     TITLE = {A stabilization-free virtual element method for the
              convection-diffusion eigenproblem},
   JOURNAL = {J. Sci. Comput.},
  FJOURNAL = {Journal of Scientific Computing},
    VOLUME = {102},
      YEAR = {2025},
    NUMBER = {2},
     PAGES = {Paper No. 46, 33},
      ISSN = {0885-7474,1573-7691},
   MRCLASS = {65N15 (65N25 65N30)},
  MRNUMBER = {4846353},
MRREVIEWER = {Shuhao\ Cao},
       DOI = {10.1007/s10915-024-02765-1},
       URL = {https://doi.org/10.1007/s10915-024-02765-1},
}

@article {Mengetal2,
    AUTHOR = {Meng, Jian and Guan, Lei and Qian, Xu and Song, Songhe and
              Mei, Liquan},
     TITLE = {Stabilization-free virtual element method for the transmission
              eigenvalue problem on anisotropic media},
   JOURNAL = {J. Comput. Math.},
  FJOURNAL = {Journal of Computational Mathematics},
    VOLUME = {44},
      YEAR = {2026},
    NUMBER = {1},
     PAGES = {103--134},
      ISSN = {0254-9409,1991-7139},
   MRCLASS = {65N25 (65N30 74S05 92E10)},
  MRNUMBER = {4994616},
       DOI = {10.4208/jcm.2410-m2024-0023},
       URL = {https://doi.org/10.4208/jcm.2410-m2024-0023},
}
\bibliographystyle{plain}

\end{document}